\documentclass{amsart}
\usepackage[margin=1.2in]{geometry}

\usepackage{amssymb}
\usepackage{amsmath}
\allowdisplaybreaks

\usepackage{xcolor}
\usepackage[all]{xy}
\usepackage{hyperref}
\usepackage{amssymb}
\usepackage{amsthm}
\usepackage{tikz-cd}
\usepackage{mathrsfs}
\usepackage{diagbox}
\usepackage{multirow}
\usepackage{float}

\newcommand{\BA}{{\mathbb {A}}}

\newcommand{\BD}{{\mathbb {D}}}
\newcommand{\BF}{{\mathbb {F}}}
\newcommand{\BQ}{{\mathbb {Q}}}
\newcommand{\BZ}{{\mathbb {Z}}}

\newcommand{\CA}{{\mathcal {A}}}
\newcommand{\CF}{{\mathcal {F}}}
\newcommand{\CM}{{\mathcal {M}}}
\newcommand{\CO}{{\mathcal {O}}}
\newcommand{\CT}{{\mathcal {T}}}

\newcommand{\fa}{{\mathfrak{a}}}
\newcommand{\fb}{{\mathfrak{b}}}
\newcommand{\fp}{{\mathfrak{p}}}
\newcommand{\fq}{{\mathfrak{q}}}

\newcommand{\Gal}{\mathrm{Gal}}
\newcommand{\Ker}{\mathrm{Ker}}

\newcommand{\Coker}{\mathrm{Coker}}

\newcommand{\rank}{\mathrm{rank}}
\newcommand{\rk}{\mathrm{rk}}

\newcommand{\Aut}{\mathrm{Aut}}
\newcommand{\Cl}{\mathrm{Cl}}
\newcommand{\cl}{\mathrm{cl}}
\newcommand{\cyc}{\mathrm{cyc}}
\newcommand{\pr}{\mathrm{pr}}
\renewcommand{\Im}{\mathrm{Im}}
\newcommand{\tF}{\widetilde{F}}

\newcommand{\leg}[2]{\left(\frac{#1}{#2}\right)}
\newcommand{\lleg}[2]{\left[\frac{#1}{#2}\right]}

\newcommand{\hilbert}[3]{\Bigl(\dfrac{#1,#2}{#3}\Bigr)}
\newcommand{\lhilbert}[3]{\left[\dfrac{#1,#2}{#3}\right]}

\numberwithin{equation}{section}

\theoremstyle{plain}
\newtheorem{thm}{Theorem}[section] 
\newtheorem{lem}[thm]{Lemma}  \newtheorem{prop}[thm]{Proposition}
\newtheorem {conj}[thm]{Conjecture} \newtheorem{defn}[thm]{Definition}
    \newtheorem{remk}[thm]{Remark}

\newtheorem*{2-rank formula}{2-rank formula}

\theoremstyle{remark} \newtheorem{remark}[thm]{Remark}
\theoremstyle{remark} 
\theoremstyle{remark}

\begin{document}
\title{On abelian $2$-ramification torsion modules of quadratic fields }

\author{Jianing Li}
\address{Research Center for Mathematics and Interdisciplinary Sciences, Shandong University, Qingdao, 266237, P. R. China }
\email{lijn@sdu.edu.cn}

\author{Yi Ouyang}
\address{CAS Wu Wen-Tsun Key Laboratory of Mathematics, University of Science and Technology of China, Hefei, Anhui 230026, China}

\email{yiouyang@ustc.edu.cn}

\author{Yue Xu}
\address{CAS Wu Wen-Tsun Key Laboratory of Mathematics, University of Science and Technology of China, Hefei, Anhui 230026, China}
\email{wasx250@mail.ustc.edu.cn}



\maketitle

\begin{abstract}
For a  number field $F$ and  a prime number $p$,  the $\BZ_p$-torsion module of the Galois group of the maximal abelian pro-$p$ extension of $F$ unramified outside $p$ over $F$, denoted as  $\CT_p(F)$, is an important subject in abelian $p$-ramification theory.
In this paper we study the group $\CT_2(F)=\CT_2(m)$ of the quadratic field $F=\BQ(\sqrt{ m})$.
Firstly, assuming $m>0$, we prove an explicit $4$-rank formula for quadratic fields that $\rk_4(\CT_2(-m))=\rk_2(\CT_2(-m))-\rank(R)$ where $R$ is a certain explicitly described  R\'edei  matrix over $\BF_2$. Furthermore, using this formula, we obtain the $4$-rank density of $\CT_2$-groups of imaginary quadratic fields. Secondly, for $l$  an odd prime,  we obtain results about the $2$-power divisibility of  orders of $\CT_2(\pm l)$ and $\CT_2(\pm 2l)$, both of which are cyclic $2$-groups. In particular we find that $\#\CT_2(l)\equiv 2\# \CT_2(2l)\equiv h_2(-2l)\pmod{16}$  if $l\equiv 7\pmod{8}$ where $h_2(-2l)$ is the $2$-class number of $\BQ(\sqrt{-2l})$. We then obtain density results for $\CT_2(\pm l)$ and $\CT_2(\pm 2l)$ when the orders are small. Finally, based on our density results and numerical data, we propose distribution conjectures about $\CT_p(F)$ when $F$ varies over real or imaginary quadratic fields for any prime $p$, and about $\CT_2(\pm l)$ and $\CT_2(\pm 2 l)$ when $l$ varies, in the spirit of Cohen-Lenstra heuristics. Our conjecture in the  $\CT_2(l)$ case is closely connected to Shanks-Sime-Washington's speculation on the distributions of the zeros of $2$-adic $L$-functions and to the distributions of the fundamental units.
\end{abstract}

\section{Introduction and Main results}\label{sec:intro}

Let $p$ be a prime number. For a number field $F$, let $M=M(F,p)$ be  the maximal abelian pro-$p$ extension of $F$ unramified outside $p$. By class field theory, $\Gal(M/F)$ is a finitely generated $\BZ_p$-module of rank $r_2(F) + \delta_p(F)+1$, where $r_2(F)$ is the number of complex places of $F$ and $\delta_p(F) \geq 0$ is the Leopoldt defect of $F$ at $p$. Leopoldt's Conjecture is that $\delta_p(F)=0$ for all $p$ and $F$ and this has been proved when $F/\BQ$ is abelian. We call the $\BZ_p$-torsion subgroup of $\Gal(M/F)$, a finite abelian $p$-group, the $\CT_p$-group of $F$ and denote it by  $\CT_p(F)$.
The study of $\Gal(M/F)$ and $\CT_p(F)$ which goes back to fundamental contributions of Serre, Shafarevich and Brumer, is the so-called abelian $p$-ramification theory. We refer the reader to the historical survey \cite{Gra19} by Gras for this theory, in which the $p$-rank formula for $\CT_p(F)$ due to himself is stated. When $F$ is totally real, assuming $\delta_{p}(F)=0$, the work of Coates \cite{Coa77} and Colmez \cite{Col88} shows that the order of $\CT_p(F)$ is essentially the residue of the $p$-adic zeta function of $F$ up to a $p$-adic unit.
This motivates us to study the group structure of $\CT_p(F)$ in more detail. Like class groups, the study of $\CT_p(F)$ can be much more explicit in the case that $F$ is a quadratic field and $p=2$. In this paper, we will mainly consider this case, and our main purpose is to study the distribution of $\CT_2(F)$ when $F$ varies in a certain family of quadratic fields.

Note that the structure of a finite abelian $p$-group $A$ is completely determined by its $p^i$-rank $\rk_{p^i}(A):=\dim_{\BF_p} p^{i-1}A/p^i A$   for all $i$. As a consequence, to study $\CT_p(F)$, it is necessary and sufficient to study $\rk_{p^i}(\CT_p(F))$ for all $i$.

The general $p$-rank formula for $\CT_p(F)$ becomes very explicit for $p=2$ and $F$ quadratic, after a computation of genus class numbers; see Theorem~\ref{thm:2-rank}. If $F$ is imaginary quadratic, we shall prove an explicit $4$-rank formula of $\CT_2(F)$, namely, $\rk_4(\CT_2(F))$ is the difference of $\rk_2(\CT_2(F))$ and the rank of a certain explicitly described R\'edei matrix; see Theorem~\ref{thm:4-rankT}. This formula is new and is analogous to the classical $4$-rank formula for narrow class groups of quadratic fields. Applying this result, we deduce the following $4$-rank density formula for $\CT_2$-groups of imaginary quadratic fields, which is the main result of this paper:

\begin{thm}[$4$-rank density formula for $\CT_2$ of imaginary quadratic fields]\label{thm:4-rank-density}	
For integers $t\geq 1$ and $r\geq 0$, and a real number $x>0$, put
\begin{align*}
&N_{x} := \{m \in \BZ_{>0}\mid  m\leq x \text{ squarefree} \},  \\
&N_{t;x}:=\{m \in N_{x} \mid \text{ exactly } t \text{ prime numbers are ramified in } \BQ(\sqrt{-m}) \},   \\
&T^r_{t;x}:=\{m \in N_{t;x} \mid  \rk_4(\CT_2(\BQ(\sqrt{-m})))=r \}.
\end{align*}		
Then for any integer $r\geq 0$, the limit $d^T_{\infty,r}$, which is defined by
\begin{equation}\label{eq:limit-definition2}
d^T_{\infty,r}:= \lim\limits_{t\rightarrow \infty} \lim\limits_{x\rightarrow \infty} \frac{\#T^r_{t;x}}{\#N_{t;x}}
%\quad \text{ and } \quad  d^T_{r}:=  \lim\limits_{x\rightarrow\infty}\frac{ \# T^r_{x}}{\# N_{x}},
\end{equation}
exists and
\begin{equation}\label{eq:4-rank-density}
d^T_{\infty,r}=\frac{\prod\limits_{i=r+2}^{\infty}(1-2^{-i})}{2^{r(r+1)}  \prod\limits_{i=1}^{r}(1-2^{-i})}=\frac{\eta_\infty(2)}{2^{r(r+1)}\eta_r(2)\eta_{r+1}(2)}
\end{equation}
where $\eta_s(q):=\prod_{i=1}^s (1-q^{-i})$ for $s\in \BZ_{>0}\cup \{\infty\}$ and $q\geq 2$ and $\eta_0(q):=1$.
\end{thm}
\begin{remark}
Theorem~\ref{thm:4-rank-density} is analogue to the density theorem of Gerth \cite{Ger84} on the $4$-rank of narrow class groups of quadratic fields, and to the theorem of Yue-Yu \cite{YY04} on the $4$-rank of the tame kernel of quadratic fields.
\end{remark}

We then turn to  study the $\CT_2$-groups of subfamilies of quadratic fields, namely $\BQ(\sqrt{\pm l})$ and  $\BQ(\sqrt{\pm 2l})$ where $l$ is an odd prime. For simplicity, write $\CT_2(m)$ for $\CT_2(\BQ(\sqrt{m}))$, $t_2(m)$ for its order, and $h_2(m)$ for the $2$-class number of $\BQ(\sqrt{m})$. Such questions for $\CT_2(l)$ and $\CT_2(2l)$ have been studied by many researchers before. For example, consider $\BQ(\sqrt{l})$ and let $L_2(1,\chi_l)$ be the $2$-adic $L$-function where $\chi_l$ is the quadratic character associated with $\BQ(\sqrt{l})$. Recalling that $h_2(l)=1$, then Coates' order formula (see \cite[Appendix 1]{Coa77} or Proposition~\ref{prop:coates}) directly relates $\# \CT_2(l)$ to the $2$-adic regulator of $\BQ(\sqrt{l})$ and therefore to the $2$-adic valuation of $L_2(1,\chi_l)$ by the class number formula.
The latter two objects and their relation to $h_2(-l)$ and $h_2(-2l)$ have been studied by Kaplan, Leonard, Williams (see \cite{KW82}, \cite{LW82}, \cite{Wil81}) and by Shanks-Sime-Washington \cite{SSW99}. However, it seems that there is no study for $\CT_2(-l)$ and $\CT_2(-2l)$ before.

By the $2$-rank formula \eqref{eq:2-rankT}, $\CT_2( \pm l ) =\CT_2( \pm 2l ) =0$ if  $l\equiv \pm 3\pmod 8$ and $\CT_2( \pm l  )$ and   $\CT_2( \pm 2l  )$ are nontrivial $2$-cyclic groups if $l\equiv \pm 1\pmod 8$. Applying our $4$-rank formula and Coates' order formula for totally real fields, we obtain the following results:
\begin{itemize}
	\item (Theorem~\ref{thm:special-family1}) Determine the congruent conditions for $l$ satisfying $t_2(-l)$ or $t_2(-2l)=2$, $4$ and $\geq 8$, and hence find the respective densities;
	
	\item (Theorem~\ref{thm:l78}) Determine the conditions for $l\equiv 7\pmod{8}$ satisfying  $t_2(l)=4, 8$ and $\geq 16$, and deduce the formula
	\begin{equation} t_2(l)\equiv 2t_2(2l)\equiv h_2(-2l)\pmod{16}. \end{equation}
	
	\item (Proposition~\ref{prop:special-fam})  Determine the conditions for $l\equiv 1\pmod{8}$ satisfying  $t_2(l)$ or $t_2(2l)=2$ or $4$.
\end{itemize}
Here Theorem~\ref{thm:special-family1} is new, Theorem~\ref{thm:l78} is an improvement of the result in \cite{LW82} and Proposition~\ref{prop:special-fam} is essentially a summary of the results in \cite{KW82}, \cite{LW82} and \cite{Wil81} using the language of $\CT_2$-groups.

For the real case, we then have the following density result which is inspired by the work on the distribution of $2$-adic valuation of $L_2(1,\chi_l)$ in \cite{SSW99}.
\begin{thm} \label{thm:i=0,1}
		For  $i\in \{0,1\}$ and $e\in \{0, 1\}$,
		\begin{equation}\label{eq:densityaab}
			\lim\limits_{x \rightarrow \infty}\frac{\#\{l \leq x: l \equiv (-1)^e \pmod 8,\ t_2(l)= 2^{i+1+e}\}}{\#\{l \leq x: l \equiv (-1)^e \pmod 8\}}=\frac{1}{2^{i+1}}.
		\end{equation}
		\begin{equation}\label{eq:density2aaa}
			\lim\limits_{x \rightarrow \infty}\frac{\#\{l \leq x: l \equiv (-1)^e \pmod 8,\ t_2(2l)=2^{i+1}\}}{\#\{l \leq x: l \equiv (-1)^e \pmod 8\}}=\frac{1}{2^{i+1}}.
		\end{equation}	
	\end{thm}

In the last section, we shall present several conjectures in light of the density results we proved in the spirit of Cohen-Lenstra heuristics. We shall  present computational evidence for our conjectures in the Appendix.

\subsection*{Acknowledgment.}{The authors are grateful to the anonymous referees for their very helpful suggestions and remarks.
The authors are partially supported by Anhui Initiative in Quantum Information Technologies (Grant No. AHY150200).}

\section{The rank and density formulas for quadratic imaginary fields }\label{sec:proof-Redei}
\subsection{Notations} We shall use the following notations.

(1) For a general number field $F$, $\CO_F$ is the ring of integers of $F$,  $\CO_F^\times$ is the group of units of $F$, $r_1$ and $r_2$ are the numbers of real and complex places of $F$,   $n=r_1+2r_2=[F:\BQ]$.
For a finite place $v$ of $F$, we let $U_v$ and $U_{1,v}$ be the groups of local units and principal local units. For $v$ infinite, let  $U_v = F^\times_v$. Let $\BA_F$ be the ad\`ele ring of $F$. The id\`elic group of $F$, as the units of $\BA_F$, is denoted by $\BA_F^\times$.

Let $F^+=\{\alpha\in F\mid v(\alpha)>0\ \text{for all real places}\ v\ \text{of}\ F \}$ be the subgroup of $F^\times$ of totally real elements. Hence $F^\times/F^+$ is an $\BF_2$-vector space of dimension $r_1$, by the approximation theorem.

Let $S=S_p$ be the set of primes of $F$ lying above $p$. Let $\CO_{S}, E_{S}, \Cl_{S}$ and $\Cl^{+}_{S}$ denote the ring of $S$-integers, the group of $S$-units, the $S$-class group, and the narrow $S$-class group of $F$, respectively.  Let $E_S^+=E_S\cap F^+$. Let  $U_{1,S}=\prod_{v\in S}U_{1,v}$.

(2) In the special case that $F$ is a quadratic field, write  $F=\BQ(\sqrt{m})$, then $(r_1,r_2)=(2,0)$ if $m>0$ and $(0,1)$ if $m<0$. Let $G=\Gal(F/\BQ)=\{1,\sigma\}$. Let $\Cl(m)$, $\Cl_p(m)$, $h(m)$, $h_p(m)$ and $\CT_p(m)$ be the class group, the $p$-class group, the class number, the $p$-class number and the $\CT_p$-group of $F=\BQ(\sqrt{m})$ respectively. Let $t_p(m)=\#\CT_p(m)$.

If $p=2$,  the size of $S$ is  $2$ if $2$ splits  and $1$ if $2$ is not split in $F$.  If $F$ is imaginary, $F^+=F^\times$ and $\sigma$ is the restriction of complex conjugation on $F$.

(3) For any abelian group $A$, $A[n]$ is the $n$-torsion subgroup of $A$ and $A[p^\infty]$ is the $p$-primary part of $A$. For a finite abelian group $A$ and a positive integer $i$, the $p^i$-rank $\rk_{p^i}(A):=\dim_{\BF_p} p^{i-1}A/p^i A$. If $A$ is an $\BF_2$-vector space, $\dim A:=\dim_{\BF_2} A$ is its dimension.

(4) For Jacobi, $2$-nd Hilbert and Artin -symbols with values in $\mu_2=\{1,-1\}$, we use $[\ ]$ instead of $(\ )$ to represent  the corresponding additive symbols with values in $\BF_2=\{0,1\}$.

\subsection{The $2$-rank and $4$-rank formulas in general}
For $F$ a general number field, we  recall some  facts about $\CT_p(F)$.
All are standard consequences of global class field theory; see, for example, \cite[Theorem 13.4]{Was97}. The closed subgroup $ \overline{F^\times \prod_{v\notin S} U_v}$ of $\BA^\times_F$ corresponds to the maximal abelian extension of $F$ unramified outside $S$. Set
\begin{equation}\label{eq:2.1}\CA_F:=\BA^\times_F\bigg/ \overline{F^\times \prod_{v\notin S} U_v}. \end{equation}
As known in the proof of \cite[Theorem 13.4]{Was97}, the induced Artin map $\CA_F \twoheadrightarrow \Gal(M/F)$
is surjective and has finite kernel of  prime-to-$p$ order, thus induces a canonical isomorphism
 \[ {\CA}^{\mathrm{pro-}p}_F\cong \Gal(M/F), \]
 where ${\CA}^{\mathrm{pro-}p}_F$ is the pro-$p$-part of $\CA_F$.
 Let $H$ be the $p$-Hilbert class field of $F$. Then $\Gal(H/F)\cong  \Cl_p(F)$ canonically.  Let   $\phi$ be the canonical diagonal embedding $F \hookrightarrow \prod_{v\in S}F_v$ and
 $E_{1,S}=\phi^{-1}(U_{1,S})\cap \CO^\times_F$.  By class field theory, the following diagram is commutative with exact rows:
\begin{equation}\label{eq:2.3}
\begin{tikzcd}
0 \ar[r] &  U_{1,S}/{\overline{\phi(E_{1,S})}} \ar[r] \ar[d, "\cong"] &{\CA}^{\mathrm{pro-}p}_F \ar[r]\ar[d, "\cong"] &  \Cl_p(F)  \ar[r] \ar[d, "\cong"] & 0 \\
0 \ar[r] &  \Gal(M/H) \ar[r] & \Gal(M/F)  \ar[r]&  \Gal(H/F)  \ar[r] & 0.
\end{tikzcd}
\end{equation}
The group $U_{1,S}$ is a finitely generated $\BZ_p$-module of rank $n=r_1+2r_2$ and the submodule ${\overline{\phi(E_{1,S})}}$ is of  rank $r_1+r_2-1-\delta_p(F)$ for some integer $\delta_p(F)\geq 0$. It follows that $\Gal(M/F)$ is a finitely generated $\BZ_p$-module of rank $r_2 +1+\delta_p(F)$. Leopoldt conjectured that $\delta_p(F)$ is always $0$ and this has been proved when $F$ is abelian over $\BQ$.  Thus $\CT_p(F)$, by definition the torsion subgroup of $\Gal(M/F)$, is finite and
\begin{equation}\label{eq:2.2}
\CT_p(F) \cong \CA_F [p^\infty],
\end{equation}
and the $p$-rank of $\CT_p(F)$ is given by
\begin{equation}\label{eq:2.4}
\rk_p( \CT_p(F) ) = \rk_p( \Gal( M/F ) ) -r_2-1-\delta_p(F).
\end{equation}

From now on,  we identify $\CA_F[p^\infty]$ with $\CT_p(F)$. By abuse of notation, we write $\CA_F$ and $\Gal(M/F)$ additively. Let $L$ be the maximal abelian extension of $F$ which is of exponent $p$ and unramified outside $p$. Then $L$ is the intermediate field of $M/F$ fixed by $p\Gal(M/F)$. The induced Artin map $\CA_F\rightarrow \Gal(M/F)$ has kernel consisting of prime-to-$p$-torsion elements, hence is contained in $p\CA_F$ and  the induced map $\CA_F/p\CA_F \rightarrow \Gal(L/F)$ is an isomorphism. The kernel of the composite map
 \[ \varphi: \CT_p(F)[p]\hookrightarrow\ \CA_F\rightarrow \CA_F/p\CA_F=\Gal(L/F) \]
is $\CT_p(F)[p]\cap p\CA_F= p\CT_{p}(F)[p^2]$, which is an $\BF_p$-space of dimension $\rk_{p^2}(\CT_p(F))$. This gives the identity
\begin{equation} \label{eq:p2rank}
	\rk_{p^2}(\CT_p(F))=\rk_p(\CT_p(F))-\dim_{\BF_p} \Im(\varphi).
\end{equation}

 We first derive the $2$ and $4$-rank formulas  of $\CT_2$ for a general number field, and the $2$-rank formula for a quadratic field. The general $2$-rank formula \eqref{eq:2-rank} was proved in Gras~\cite[Th\'{e}or\`{e}me I 3]{Gra82}, and the $4$-rank formula is quite routine.
\begin{thm} \label{thm:2-rank} Let $F$ be a number field. Let $S$ be the set of primes in $F$ above $2$ and $\Cl^+_S$ the narrow $S$-class group of $F$.
	
	$(1)$ $($Gras$)$ The $2$-rank of $\CT_2(F)$ is given by the formula
	\begin{equation}\label{eq:2-rank}
		\rk_2 \CT_2(F)=\#S + \rk_2(\Cl^{+}_{S}) -1-\delta_2(F).
	\end{equation}
	In particular, if $m$ is a squarefree integer with $t$ odd prime factors, then for $F=\BQ(\sqrt{m})$,
	\begin{equation}\label{eq:2-rankT}
		\rk_2(\CT_2(F))  =
		\begin{cases}
			t & \text{ if }\ q\equiv \pm 1\pmod{8}\ \text{for all odd prime}\ q\mid m,\\
			t-1 & \text{ if }\ q\equiv \pm 3\pmod{8}\ \text{for some odd prime}\ q\mid m.
		\end{cases}
	\end{equation}
	
	$(2)$ Suppose $A$ is a finite set of id\`eles which generates $\CT_2(F)\subset \CA_F:= \BA_F^\times\big/\overline{F^\times \prod_{v\nmid 2} U_v}$. Suppose $B$ is a finite set of elements in $F^\times$ such that $F(\sqrt{B})$ is  the maximal abelian extension of $F$ of exponent $2$ unramified outside $2$. For $a\in A$ and $b\in B$, let $[a,b]=\log_{-1} (a,F(\sqrt{b}))\in \BF_2$ be the additive Artin symbol. Let $R=([a,b])_{a\in A, b\in B}$. Then
	\begin{equation} \label{eq:4-rank}
		\rk_4 (\CT_2(F)) =  \rk_2 (\CT_2(F)) - \rank(R).
	\end{equation}
\end{thm}
\begin{remark}\label{rmk:2.3}
	(1) The minimal size of $A$ is $\rk_2(\CT_2(F))$, and the minimal size of $B$ is $\rk_2(\Gal(M/F)))=\rk_2(\CT_2(F))+r_2(F)+1+\delta_2(F)$. Moreover, if $F(\sqrt{b})$ is contained in a $\BZ_2$-extension of $F$, then $[a,b]=0$ for all $a\in A$ and we can delete the corresponding row in $R$.
	
	(2) The  $p^2$-rank formula for $\CT_p(F)$ in the case  $\mu_p\subseteq F$ can be proved similarly,  as the kernel of the map $ \CT_p(F)[p]\hookrightarrow  \CA_F\rightarrow \CA_F/p\CA_F\cong \Gal(L/F)$  is $p\CT_p[p^2]$. Moreover, one can similarly deduce the formula $\rk_{p^{i+1}}(\CT_p(F))=\rk_{p^{i}}(\CT_p(F))-\dim_{\BF_p} \Im(\CT_p(F)[p^i]\rightarrow \Gal(L/K))$.
\end{remark}

\begin{proof}
We are in the case $p=2$. Then 	$L$ is the maximal abelian extension of $F$ of exponent $2$ unramified outside $S$. By Kummer Theory, $L=F(\sqrt{J})$, where $J$ is the finite subgroup of $F^\times/F^{\times 2}$ given by
	\begin{equation}\label{eq:J}
		J:=\{ \beta\in F^{+} \mid  \beta\CO_{S}=\fb^2 \text{ for some } \CO_S \text{-fractional ideal } \fb \text{ of  } F \}/(F^\times)^2.
\end{equation}	
		
(1) First suppose $F$ is general.  The non-degeneracy of the Kummer pairing $J \times \Gal(L/F)\rightarrow \{ \pm 1 \}$ then implies
\begin{equation}\label{eq:2.6}
\rk_2  \Gal(M/F  ) =\dim \Gal(L/F) = \dim J.
\end{equation}
Let  $\pr$ be the natural projection $\Cl^{+}_{S}\rightarrow \Cl_{S}$  and
$\Cl_{S,+}=\pr(\Cl^{+}_{S}[2])\subset \Cl_S[2]$. For $[\beta] \in J$, $\beta\CO_{S}=\fb^2$, then the class map $\cl_S(\fb)$ lies in $\Cl_{S,+}$.  This gives an exact sequence of $\BF_2$-vector spaces:
\begin{equation}\label{eq:2.7}
1 \rightarrow E^{+}_{S}/E^2_{S} \rightarrow J \xrightarrow[]{\beta\mapsto \cl_S(\fb)} \Cl_{S,+}\rightarrow 1.
\end{equation}
Let $F^\times \CO_{S}=\{ \alpha \CO_{S}\mid  \alpha \in F^\times\}$ and $F^+\CO_{S}=\{ \alpha \CO_{S}\mid  \alpha \in F^+\}$, then  $\ker \pr=F^\times \CO_{S}/{F^+\CO_{S}} \subset \Cl^+_S[2]$.  This gives an exact sequence  of $\BF_2$-vector spaces
\begin{equation} \label{eq:cls+}
 1\rightarrow F^\times \CO_{S}/{F^+\CO_{S}}\rightarrow \Cl^{+}_{S}[2]\rightarrow \Cl_{S,+}\rightarrow 1.
\end{equation}
We also have the following natural exact sequence  of $\BF_2$-vector spaces:
\begin{equation} \label{eq:totalpos}  1\rightarrow  E_{S}/E^{+}_{S}\rightarrow  F^\times/F^+ \rightarrow  F^\times \CO_{S}/{F^+\CO_{S}}  \rightarrow 1.  \end{equation}
Combining the above results, we get
\[
\begin{split}
\rk_2 \Gal(M/F) = & \dim E^{+}_{S}/E^2_{S}+ \dim  \Cl_{S,+}  \\
= &  \dim E^{+}_{S}/E^2_{S}+\dim \Cl^{+}_{S}[2]- r_1 + \dim E_{S}/E^{+}_{S} \\
= & \dim E_{S}/E^2_{S} + \dim \Cl^{+}_{S}[2]- r_1\\
= & r_2 +\#S + \dim  \Cl^{+}_{S}[2],
\end{split}
\]
where $\dim F^\times/F^+=r_1$ by the approximation theorem, and $\dim E_S/E_S^2=r_1+r_2+\#S$ by  Dirichlet's unit theorem that $E_{S} \cong \BZ^{r_1+r_2+\#S-1}\times \BZ/{d\BZ}$ with $d$ even. By \eqref{eq:2.4}, we then get the general $2$-rank formula \eqref{eq:2-rank} for $\CT_2$-group of a general base field (see \cite{Gra82} for a slightly different approach).

 Now suppose $F=\BQ(\sqrt{m})$ is a quadratic field. Then $\delta_2(F)=0$. Write $G=\Gal(F/\BQ)$.  Since $\BQ$ has class number $1$, we conclude that $\Cl^{+}_S[2]=(\Cl^{+}_S)^G$.  Recall that $t$ is the number of odd prime factors of $m$. Applying the $S$-narrow version of the ambiguous class number formula (see, for example \cite[Remark 4.5]{LY20}) gives the following result:
\begin{equation}\label{eq:S-clgp}
\dim (\Cl^{+}_{S}) ^G =
\begin{cases}
t-2,  &\text{ if } 2 \text{ splits  and } 2 \notin N(F); \\
t-1,   &\text { if } 2 \text{ splits and } 2\in N(F) \text{ or } 2\text{ does not split and } 2 \notin N(F);\\
t,  &\text{ if } 2\text{ does not split  and } 2 \in N(F).
\end{cases}
\end{equation}
By Lemma~\ref{lem:hasse} below, $2\in N(F)$ if and only if $q\equiv \pm 1\pmod{8}$ for all odd primes $q\mid m$, the $2$-rank formula \eqref{eq:2-rankT}  fo $F=\BQ(\sqrt{m})$ then follows.

(2) We may assume $\widetilde{B}=\{b\pmod{F^{\times 2}}\mid b\in B\}$ is an $\BF_2$-basis of $J$. Then
\[
\Gal(L/F)\hookrightarrow \prod\limits_{b\in B}\Gal(F(\sqrt{b})/F)
\]
is an isomorphism. Written additively, the map $\varphi$ sends $a\in \CT_2(F)[2]\subset \CA_F$ to $([a,F(\sqrt{b})])_{b\in B}$. Thus $\dim_{\BF_2}(\Im(\varphi))$ is nothing but the rank of $([a,b])_{a\in A, b\in B}$. By \eqref{eq:p2rank}, we get the $4$-rank formula.
\end{proof}

We have the following easy lemma to transform the norm conditions into congruent conditions.
\begin{lem}\label{lem:hasse}
Let $m$ be a positive squarefree integer.
Let $F=\BQ(\sqrt{-m})$ and  $\tF=\BQ(\sqrt{m})$. Then
 \begin{align*}
 	2& \in N(F)\Longleftrightarrow\  2\in N(\tF)  \Longleftrightarrow\   q\equiv \pm 1 \pmod 8\ \text{for all odd prime}\ q\mid m;\\
 	 -2 & \in N(\tF) \Longleftrightarrow\  q\equiv  1,\ 3 \pmod 8\ \text{for all odd prime}\ q\mid m;\\
 	-1 & \in N(\tF)  \Longleftrightarrow\  q\equiv 1 \pmod 4\ \text{for all odd prime}\ q\mid m. 	
 \end{align*}
\end{lem}
\begin{proof}
By Hasse's norm theorem and the product formula, $2\in N(F)$ if and only if $2\in N(F_v)$ for all but one prime $v$ of $F$. If $v\nmid 2m$, then  $v$ is always unramified and $2\in N(F_v)$ by local class field theory. For an odd prime $q\mid m$, $q$ is ramified in $F$. Let $\fq$ be the unique ramified prime of $F$ above $q$. Then $2\in N(F_\fq)$ if and only if the Hilbert symbol $(2,-m)_{q}=1$, which is equivalent to that $q\equiv \pm 1\pmod 8$. If $2$ splits in $F$, then $v\mid 2$ is unramified and $2\in N(F_v)$; in other cases, there is only one prime $v$ above $2$ which can be excluded from consideration. Hence $2\in N(F)$  if and only if $q\equiv \pm 1 \pmod 8$  for every odd prime $q\mid m$. The other cases can be proved similarly.
\end{proof}
\subsection{The explicit $4$-rank formula for imaginary quadratic fields}
We turn to work on the imaginary quadratic field case. We shall work out $A$ and $B$ explicitly for   an imaginary quadratic field and hence obtain an explicit $4$-rank formula in this case. This explicit formula  will be used to deduce the $4$-rank density formula of $\CT_2$-groups of imaginary quadratic fields in next subsection.

We suppose $m>0$ and $F=\BQ(\sqrt{-m})$. Let $\{q_1,\cdots, q_t\}$ be the set of odd prime factors of $m$, arranged in such a way  that $q_i\equiv \pm 1\pmod 8$ if $1\leq i \leq k$ and $\pm 3\pmod 8 $ if $k<i\leq t$. Note that $k=0$ if $q\equiv \pm 3\pmod 8$ for all $q\mid m$. Let $\fp$ be a prime of $F$ above $2$. Then $\fp$ is either the unique prime above $2$ or $(2)=\fp \bar{\fp}$ splits in $F$ where $\bar{\fp}\neq \fp$ is the complex conjugate of $\fp$. Let $\fq_i$ be the unique prime of $F$ above $q_i$. For an odd prime $q$, let  $q^* = (-1)^{(q-1)/2}q$. Then $q^*_i\ (1\leq i\leq k)$ and $q^*_j q^*_{j'}\ (k<j,j'\leq t)$ are squares in the $2$-adic field $\BQ_2$.

Our explicit $4$-rank formula for $\CT_2(\BQ(\sqrt{-m}))$ is
\begin{thm}  \label{thm:4-rankT} Suppose $F=\BQ(\sqrt{-m})$. For $0\leq i\leq t$, we define the  id\`eles $a_i=(a_{i,v})\in \BA^\times_F$  as follows:
	\begin{enumerate}
		\item $a_{0,\fp}=\sqrt{-1}$ if $F_\fp \cong  \BQ_2(\sqrt{-1})$, and  $a_{0,\fp}=-1$  if $2=\fp\bar{\fp}$ splits in $F$;
		
		\item if $1\leq i\leq k$, $a_{i, \fq_i}=\sqrt{-m}$ and $a_{i,v}=\sqrt{q_i^*}$ for $v\mid 2$;
		
		\item if $k<i<t$, $a_{i, \fq_i}=a_{i,\fq_t}=\sqrt{-m}$ and  $a_{i,v}=\sqrt{q_i^* q_t^*}$ for $v\mid 2$;
		
		\item for all other places $v$, $a_{i,v}=1$. In particular, $a_t=1$	if $k<t$.
	\end{enumerate}
	Let $\pi$ be a generator of $\fp^\lambda$ where $\lambda$ is the order of $\fp$ in the class group of $F$.  If $2$ is a norm of $F$, noting that $m$ is a norm of $\BZ[\sqrt{2}]$, write $m=2g^2-h^2$ with $g,h\in \BZ_{>0}$ and define
	\begin{equation}\label{eq:alpha}
		\alpha =
		\begin{cases}
			h+\sqrt{-m}, & \mbox{ if } 2\in N(F)\setminus N((\CO_F\left[\frac{1}{2}\right])^\times), \\
			1, & \mbox{otherwise}.
		\end{cases}
	\end{equation}	
	Let
	\begin{equation}\label{eq:AB}
		A=\{ a_0, \cdots, a_t\} \subset \BA^\times_F, \quad \quad  B = \{-1, q_1, \cdots, q_t, \pi, \alpha\} \subset F^\times.
	\end{equation}
	Then $A$ and $B\cup\{2\}$ satisfy the assumptions in Theorem~\ref{thm:2-rank}(2), and $[a,2]=0$ for $a\in A$. Hence
 \begin{equation}
 	\rk_4 (\CT_2(F)) =  \rk_2 (\CT_2(F)) - \rank(R)\ \text{where}\ R=\left([a,b]\right)_{a\in A, b\in B}.
 \end{equation}	
\end{thm}
\begin{remark}
	For $F$ a general real quadratic field, it is still quite easy to find $B$, but the harder part is to find a set of generators $A$ for $\CT_2(F)[2]$. One reason is that it is not known how to obtain a system of explicit generators $\Cl(F)[2]$ for an arbitrary real quadratic field $F$ by a general formula.
\end{remark}

If $t=1$, then $F=\BQ(\sqrt{-1})$ or $F=\BQ(\sqrt{-2})$. In this case Theorem~\ref{thm:4-rankT} can be verified directly. We shall assume $t> 1$ in what follows. For an ideal $\fa$ of $F$, let $\cl( \fa)$ be its ideal class in $\Cl(F)$, and  $\cl_S(\fa)$ be its class in the $S$-class group $\Cl_S$ of $F$.

Theorem~\ref{thm:4-rankT} is then a consequence of the following three propositions.

\begin{prop}\label{prop:selmer}	
	Let $L$ be the maximal abelian extension of exponent $2$ over $F$, unramified outside $S$. Then $L=F(\sqrt{B'})$ where $B'=B\cup \{2 \}=\{-1, 2, q_1,\cdots, q_t, \pi,\alpha \}$.
\end{prop}

\begin{proof}
We include the proof, which is routine, for lack of exact references.
	Let $J'$ be the subgroup of $F^\times/(F^\times)^2$ generated by $B'$. It suffices to show
	 $J'=J$ with $J$ defined in \eqref{eq:J}.
	
	 We note that for all $x\in B'$, $x\neq \alpha$, $F(\sqrt{x})/F$ is unramified outside $S$. Thus if one can show that $F(\sqrt{\alpha})/F$ is unramified outside $S$, then $J'\subseteq J$.
	
		\medskip
		
	Suppose first that either $2\in N(E_S)$ or $2\notin N(F)$. In this case $\alpha=1$ and hence $J'\subset J$. We shall use the exact sequence \eqref{eq:2.7} to show that  $J'$ is indeed equal to $J$. Since $F$ is imaginary, $F^+=F^\times$. \eqref{eq:2.7} becomes the following exact sequence:
	\begin{equation}\label{eq:2.12}
	1 \rightarrow E_{S}/E^2_{S} \rightarrow J \xrightarrow[]{g} \Cl_{S}[2]\rightarrow 1.
	\end{equation}
	Here we recall that the map $g$ sends $\beta $ to $\cl_S(\fb)$, for $\beta \in J$ satisfying $\beta \CO_S=\fb^2$ for some $\CO_S$-fractional ideal $\fb$. Clearly $E_S/E^2_S \subset J'$, as $E_S$ is generated by $-1,2$ and $\pi$. Thus, in order to prove $ J'=J$, it suffices to show that $g(J')=\Cl_S[2]$.
	Let $G=\Gal(F/\BQ)$. Then $\Cl^G_S=\Cl_S[2]$. Let $I_S$ be the subgroup of fractional ideals of $F$ which is generated by prime ideals not in $S$. There is an isomorphism  (see \cite[Section 4]{LY20})
	\begin{equation}\label{eq:genus1}
	\Coker \left(I^G_S \rightarrow \Cl^G_S\right) \cong  \left(\BZ\left[\frac{1}{2}\right]\right)^\times\cap N(F^\times)/N(E_S).
	\end{equation}
	Since $-1\notin N(F)$ as $F$ is imaginary, the assumption that either $2\in N(E_S)$ or $2\notin N(F^\times)$ precisely implies that the group on the right hand of \eqref{eq:genus1} is trivial. Thus $\Cl^G_S$ is generated by $I^G_S$. But $I^G_S$ is generated by the ramified primes (see \cite[Lemma~4.4]{LY20}), it follows that $\Cl^G_S=\langle \fq_1,\cdots, \fq_t\rangle$. Since $g(q_i)=\cl_S(\fq_i)$ for each $i$, this proves $ g(J' )=\Cl^G_S=\Cl_S[2]$. Therefore, we have $J'= J$ when either $2\in N(E_S)$ or $2\notin N(F)$.
	
	\medskip
	
	Suppose next that $2\in N(F)$ but $2\notin N(E_S)$. By Lemma~\ref{lem:hasse}, $q_i\equiv \pm 1\pmod 8$ for $1\leq i\leq t$. Hence we can write $m=2g^2-h^2$ for some $g,h\in \BZ_{>0}$. In this case,  $\alpha=h+\sqrt{-m}$ (see \eqref{eq:alpha}). Then $\alpha+\bar{\alpha}=2h$ and $\alpha\bar{\alpha}=2g^2$ where $\bar{\alpha}$ is the complex conjugate of $\alpha$. Clearly $\gcd(g,h)=1$. It follows that $\gcd((\alpha), (\bar{\alpha}) )\mid 2\CO_F$.
	
	(1) If $m\equiv 1\pmod 8$ or $2\mid m$; then $2\CO_F=\fp^2$ is ramified in $F$ and $g$ is odd. In this case, $\fp\mid (\alpha)$ but $2\nmid (\alpha)$, otherwise $4\mid \alpha\bar{\alpha}=2g^2$. Hence $\bar{\fp}=\fp\mid (\bar{\alpha})$ and
	 $\gcd((\alpha), (\bar{\alpha}) )=\fp$. Since the integral ideals $(\alpha) \fp^{-1}$ and $(\bar{\alpha}) \fp^{-1}$ are coprime to each other and their product is a square, hence
   there exists an $\CO_F$-integral ideal $\fa$ such that
	$(\alpha)=\fp\fa^2$.

	(2) If $m\equiv 7\pmod 8$, then $2\CO_F=\fp\bar{\fp}$ splits and $g$ is even. Without loss of generality, we may assume $\fp\mid \alpha$. Then $\bar{\fp}\mid \bar{\alpha}$ and hence $\bar{\fp}\mid \alpha=2h-\bar{\alpha}$. This means $2\mid \alpha$ and $\gcd((\alpha), (\bar{\alpha}) )= 2\CO_F$. Now $\frac{\alpha}{2}\cdot \frac{\bar{\alpha}}{2}= 2\cdot (\frac{g}{2})^2$, then one and only one of $\fp$ and $\bar{\fp}$ divides $ \frac{\alpha}{2}$. Assume $\fp\mid \frac{\alpha}{2}$. Then the two integral ideals $(\alpha/2) \fp^{-1}$  and $(\bar{\alpha}/2) \bar{\fp}^{-1}$ are coprime and their product is a square,
	hence there exists an $\CO_F$-integral ideal $\fa$ such that $(\alpha)=2\fp\fa^2$.
	
	Thus, in both cases, we have
	\begin{equation}\label{eq:alpha3}
	\alpha \CO_S = \fa^2 \CO_S.
	\end{equation}
	This shows that $F(\sqrt{\alpha})/F$ is unramified outside $S$. Hence $J'\subset J$. Following the same argument in the previous case and applying \eqref{eq:2.12}, to show $J'=J$, we just need to show $g(J')=\Cl_S[2]=\Cl_S^G$.
	 If we can prove  $\Cl^G_S=\langle \cl_S(I^G_S), \cl_S(\fa)\rangle$, by the fact $\cl_S(I_S^G)\subset g(J')$ and $\cl_S(\fa)=g(\alpha)\in g(J')$, then we are done.

	We are left to prove the claim $\Cl^G_S=\langle \cl_S(I^G_S), \cl_S(\fa)\rangle$. By the isomorphism \eqref{eq:genus1} and by our assumption $2\in N(F)\setminus N(E_S)$, we have $[\Cl^G_S: \cl_S(I^G_S)]=2$. Thus we just need to show  $\cl_S(\fa)\notin \cl_S( I^G_S )$. Suppose, on the contrary, $\cl_S(\fa)\in \cl_S(I^G_S)$. Then we would have $\cl(\fa) \in \langle \cl(I^G_S),  \cl(\fp), \cl(\bar{\fp})\rangle$, since by definition $\Cl_S=\Cl_F/\langle \cl(S)\rangle$. Also note that $\cl(\bar{\fp})=\cl(\fp)^{-1}$. So we can write $\cl(\fa) = \cl(\fp)^{r_0} \prod_{i}\cl(\fq_i)^{r_i}$ for some integers $r_i\in \BZ$. Then, $\cl(\fa)^2 = \cl(\fp)^{2r_0}$. But, we have shown that $\cl(\fa)^2=\cl(\fp)^{-1}$. Hence $\fp^{2r_0+1}$ would be principal, say $\fp^{2r_0+1}=(\gamma)$. This would imply that $2=N(\gamma/2^{r_0})\in N(E_S)$ which contradicts to our assumption $2\in N(F^\times)\setminus N(E_S)$. This proves the claim. \end{proof}

\begin{lem}\label{lem:Cl2}
	If $m\equiv 3\pmod{4}$, then $\{\cl(\fq_1),\cdots, \cl(\fq_{t-1})\}$ is a basis of the $\BF_2$-vector space $\Cl(F)[2]$. If $m\equiv 1\pmod{4}$, then $\{\cl(\fp), \cl(\fq_1),\cdots, \cl(\fq_{t-1})\}$ is a basis of $\Cl(F)[2]$. If $m\equiv 2\pmod{4}$, then $\{\cl(\fq_1),\cdots, \cl(\fq_{t})\}$ is a basis of $\Cl(F)[2]$.
\end{lem}

\begin{proof}
The proof is the classical genus theory and we refer to \cite[Theorem 6.1]{Cox13} for the details.
 \end{proof}

\begin{prop}\label{prop:CT2}
	Let $\hat{A}$ be the image of $A$ in $\CA_F$. Then $\CT_2(F)[2]=\hat{A}$.
\end{prop}

\begin{proof}
	 For each $i$, $a^2_i$ is clearly in $\overline{F^\times \prod_{v\notin S} U_v}$, hence $\hat{a}_i$, the image of $a_i$ in $\CA_F$, is in $\CA_F[2]= \CT_2(F)[2]$, and $\hat{A}\subseteq \CT_2(F)[2]$. We have the following exact sequence of $\BF_2$-vector spaces induced from \eqref{eq:2.3}:
	\begin{equation}\label{eq:2.10}
		0 \longrightarrow  U_{1,S}/{\overline{\phi(E_{1,S})}}[2] \longrightarrow \CT_2(F)[2]\stackrel{f}{\longrightarrow}  \Cl(F)[2].
	\end{equation}
	Since $E_{1,S}=\{ \pm 1 \}$, the first term of \eqref{eq:2.10} has order $2$ and is generated by $\hat{a}_0$ if $F_\fp=\BQ_2$ or $\BQ_2(\sqrt{-1})$, and is trivial otherwise. Thus $ \dim \Ker(f)= \dim \Ker(f|_{\hat{A}})=1$ if $F_\fp=\BQ_2$ or $\BQ_2(\sqrt{-1})$, and $0$ if otherwise. By definition, $f(\hat{a}_i)=\cl(\fq_i)$ if $1\leq i \leq k$ and $f(\hat{a}_j)=\cl(\fq_j)\cl(\fq_t)$ if $k< j< t$.
	
	\medskip
	
	Suppose first that $m\equiv 2\pmod{4}$. In this case  $F_\fp$ can not be $\BQ_2$ or $\BQ_2(\sqrt{-1})$, so  $\Ker(f)=0$ and  $\dim (\hat{A})=\dim f(\hat{A})$.
	If $t=k$, then  $\dim f(\hat{A})=t$ by Lemma~\ref{lem:Cl2}.  Then $\CT_2(F)[2]=\hat{A}$ by the $2$-rank formula \eqref{eq:2-rankT} for $\CT_2(F)$. If $t>k$, one can	write
	\[ (f(\hat{a}_1), \cdots, f(\hat{a}_k), f(\hat{a}_{k+1}\hat{a}_t ),\cdots, f(\hat{a}_{t-1}\hat{a}_t ))= (\cl(\fq_1), \cdots, \cl(\fq_{t})) M, \]
	where $M$ is a matrix of rank $t-1$. Note that $\{\cl(\fq_1), \cdots, \cl(\fq_{t} )\}$ is an $\BF_2$-basis of  $\Cl(F)[2]$ by Lemma~\ref{lem:Cl2}, then	 $\dim\hat{A}=\dim f(\hat{A})=\rank(M)=t-1$. However, $\dim \CT_2(F)[2]=t-1$ by \eqref{eq:2-rankT} if $t>k$, hence  $\CT_2(F)[2]=\hat{A}$.
	
	\medskip
	
	Suppose next that $m\equiv \pm 1\pmod 8$. Then $t-k$ is even and $F_\fp=\BQ_2$ or $\BQ_2(\sqrt{-1})$. If $t=k$, it follows from Lemma~\ref{lem:Cl2} that $\dim f(\hat{A})=t-1$ and hence $\dim \hat{A}=t$ which coincides with $\dim \CT_2(F)[2]$ by \eqref{eq:2-rankT}. If $t-k$ is positive and even,  this time we can write
		\[ (f(\hat{a}_1), \cdots, f(\hat{a}_k), f(\hat{a}_{k+1}\hat{a}_t ),\cdots, f(\hat{a}_{t-1}\hat{a}_t ))= (\cl(\fq_1), \cdots, \cl(\fq_{t-1})) M, \]
		where $M$ is a matrix of rank $t-2$. Note that $\{\cl(\fq_1), \cdots, \cl(\fq_{t-1} )\}$ is linearly independent by Lemma~\ref{lem:Cl2}, then	 $\dim\hat{A}=\dim f(\hat{A})+1=\rank(M)+1=t-1$, which coincides with $\dim \CT_2(F)[2]$ by the $2$-rank formula \eqref{eq:2-rankT}. This proves $\CT_2(F)[2]=\hat{A}$ when  $m\equiv \pm 1\pmod 8$.
	
	\medskip
	
	Finally, suppose that $m\equiv \pm 3\pmod 8$. It follows that $t-k$ is an odd integer and the local field $F_\fp$ can not be $\BQ_2$ or $\BQ_2(\sqrt{-1})$. Then
			\[ (f(\hat{a}_1), \cdots, f(\hat{a}_k), f(\hat{a}_{k+1}\hat{a}_t ),\cdots, f(\hat{a}_{t-1}\hat{a}_t ))= (\cl(\fq_1), \cdots, \cl(\fq_{t-1})) M, \]
	where $M$ is a matrix of rank $t-1$.  Thus $\dim \hat{A}=\dim f(\hat{A})=t-1$,  which coincides with $\dim \CT_2(F)[2]$ by the $2$-rank formula \eqref{eq:2-rankT}. This proves $\CT_2(F)[2]=\hat{A}$ when  $m\equiv \pm 3\pmod 8$. \end{proof}

\begin{prop}\label{prop:a2=0}
$[a,2]=0$ for all $a\in A$.	
\end{prop}

\begin{proof}
Since $F(\sqrt{2})$ is the first layer of the cyclotomic $\BZ_2$-extension of $F$, the proposition then follows from Remark~\ref{rmk:2.3}(1). \end{proof}

\subsection{$4$-rank density formula}
The aim of this subsection is to prove Theorem~\ref{thm:4-rank-density}. We first give a simplification of the matrix $R$ in Theorem~\ref{thm:4-rankT} when $2$ is not a norm of $F=\BQ(\sqrt{-m})$. Although only the result in the case $m\equiv 3\pmod 4$ will be used in the proof of Theorem~\ref{thm:4-rank-density}, we also present the simplification in the case $m\equiv 1,2\pmod 4$, for completeness.

\begin{thm}\label{thm:rk4-T2(-m)}
Let $F=\BQ(\sqrt{-m})$, where $m$ is a positive squarefree integer. Let $q_1, q_2, \cdots q_t$ be all the ramified prime numbers in $F$ and assume $q_1=2$ if $2$ is ramified in $F$. Set
 \[ R^{C}:=\left(\lhilbert{q_i}{-m}{q_j}\right)_{2\leq i,j\leq t}\in M_{t-1}(\BF_2) \]
and
 \[\tau:=\begin{cases}
    \Big(\lleg{-2} {q_2},\cdots, \lleg{-2}{q_t}\Big)^T, & \mbox{if } m\equiv 3\pmod 8 \\
    \Big(\lleg{2} {q_2},\cdots, \lleg{2}{q_t}\Big)^T, & \mbox{otherwise}.
  \end{cases}\]
If $2\notin N(F)$, then
 \begin{equation}
 	\rk_4 \CT_2(F)=t-1-\rank~(\tau, R^{C}).
 \end{equation}
\end{thm}

\medskip

\begin{remark}
A word on the notation: Note that in the above theorem, $q_1, \cdots, q_t$ denote the ramified primes in $F$ rather than the odd prime factors of $m$ as used in Theorem~\ref{thm:4-rankT} and in last subsection. Clearly, this makes no difference when $m\equiv 3\pmod 4$.
\end{remark}

\begin{remark}\label{rmk:redei_cl}
Recall that (see \cite[\S 2]{LY20} for example) the classical R\'edei matrix for $\Cl_F$ is
 \[ R^{\Cl}:=\left(\lhilbert{q_i}{-m}{q_j}\right)_{1\leq i,j\leq t} \quad  \text{ and } \quad \rk_4(\Cl_F) =  t-1 - \rank~R^{\Cl}.\]
The matrix $R^C$ defined above is obtained from $R^{\Cl}$ by deleting its first row and  first column.
When $m\equiv 2,3 \pmod 4$, using the quadratic reciprocity law, one sees that the sums of each row and of each column of $R^{\Cl}$ are zero, hence $\rank~R^{C} = \rank~R^{\Cl}$. Therefore
  \[ \rk_4 \Cl_F=t-1-\rank~R^{C} \quad \text{ if } m\equiv 2,3\pmod 4. \]
\end{remark}

\begin{proof}
Firstly, we consider the case that $2$ is unramified, i.e., $m\equiv 3\pmod 4$. Then $\rk_2 (\CT_2(F))=t-1$ by Theorem~\ref{thm:2-rank}.

$(1)$ Assume first $m\equiv 3\pmod 8$. Then $2$ is inert in $F$.
Note that
\begin{equation}\label{eq:sum_tau1}
\sum\limits_{i=1}^{t}\lleg{-2}{q_i}= \lleg{-2}{m}= 0.
\end{equation}
Hence we can rearrange $\{q_1,\cdots, q_t\}$ without change the rank of $(\tau, R^{C})$. In this case, the sets $A$ and $B$ of Theorem~\ref{thm:4-rankT} are as follows:  $a_0=a_t=1$, $\alpha=1$ and $\pi=2$. But by Proposition~\ref{prop:a2=0}, $[a,2]=0$ for each $a\in A$. So we may assume that $A=\{a_1, \cdots, a_{t-1} \}$, $B=\{-1:=q_0, q_1,\cdots, q_t \}$.
Clearly, $B$ can be replaced by $\{-1:=q^*_0, q_1^*,\cdots, q_t^* \}$ as they generate the same group.

For $1\leq i\leq t$, note that $\sqrt{q^*_i} \in F_\fp = \BQ_2(\sqrt{5})$ and define
\[
a'_i:=(\cdots,  \underset{\fp}{\sqrt{q^*_i}},\cdots, \underset{\fq_i}{\sqrt{-m}}, \cdots ) \in \BA^\times_F.
\]
Then we have $a_i=a'_i $ for $1\leq i\leq k$ and $a_j=a'_j a'_t$ for $k< j \leq t-1$. Since $m\equiv 3\pmod 8$, $t-k$ must be odd. Then a direct computation shows
\[
a_1 \cdots a_{t-1} \equiv a'_{t} \left(\bmod  \left(\BA^{\times 2}_F, \overline{F^\times \prod_{v\notin S} U_v}\right) \right).
\]
It follows that we may replace $A$ by $\{a'_1, \cdots, a'_t\}$ as they generate the same group in $\CT_2(F)$. Therefore, by Theorem~\ref{thm:4-rankT}, we have
\[
 \rk_4 \CT_2(F)  = t -1  - \rank~([a'_i, q^*_j])_{1\leq i\leq t, 0\leq j \leq t}.
\]
Using the quadratic reciprocity law, for $i,j\geq 1$, one checks that
\[
 [a'_i,-1] = \lleg{-2}{q_i}   \quad \text{and}\quad  [a'_i,~q_j^*]= \lhilbert{m}{q_j^*}{q_i}=\lhilbert{q_i}{-m}{q_j}.
\]
 By the row-sum-zero and column-sum-zero property of the matrix mentioned in Remark~\ref{rmk:redei_cl} and the equation \eqref{eq:sum_tau1}, we conclude that
 \[ \rk_4 \CT_2(F) =t-1-\rank~(\tau, R^{C}).\]

(2) Assume next $m\equiv 7\pmod 8$. The prime $2$ splits in $F$. Note that $t>k$  since $2\notin N(F)$, hence the element $a_t\in A$ is trivial. We still replace $B$ by $\{-1:=q_0, \pi, q_1^*,\cdots, q_t^* \}$. Also note that both $a_0\in A$ and $\pi \in B$ are nontrivial. We may choose the sign of $\pi$ such that $\lhilbert{\pi}{-1}{\fp}=0$. Then $[a_0, \pi]=0$. The matrix $R$ for $\CT_2(F)$ in Theorem~\ref{thm:4-rankT} is
\begin{align}\label{eq:matrix7mod8}
&R=\left(\begin{array}{ccccc}
	1   & 0    & \cdots & 0  & \cdots    \\
	\lleg{-1}{q_1}   & \lhilbert{\sqrt{-m}}{\pi}{\fq_1}  & \cdots & \lhilbert{m}{q_j^*}{q_1}  & \cdots    \\
	\vdots  & \vdots   &   & \vdots &    \\
	\lleg{-1}{q_tq_{k+1}}   & \lhilbert{\sqrt{-m}}{\pi}{ \fq_{k+1}} + \lhilbert{\sqrt{-m}}{\pi}{\fq_t}   & \cdots & \lhilbert{m}{q_j^*}{q_{k+1}}+ \lhilbert{m}{q_j^*}{q_{t}}  & \cdots   \\
	\vdots  & \vdots   &   & \vdots &     \\
	\lleg{-1}{q_t q_{t-1}}   & \lhilbert{\sqrt{-m}}{\pi}{\fq_{t-1}} + \lhilbert{\sqrt{-m}}{\pi}{\fq_t}  & \cdots & \lhilbert{m}{q_j^*}{q_{t-1}}+ \lhilbert{m}{q_j^*}{q_{t}} & \cdots \end{array}\right).
\end{align}
We make the following elementary operations on the matrix $R$:
Firstly, replace the first column by $(1,0,\cdots, 0 )^T$, and replace the first row $(\cdots)$ by $\left(1, \scriptstyle\lhilbert{\scriptstyle\sqrt{-m}}{\scriptstyle\pi}{\scriptstyle\fq_t}, \cdots,  \scriptstyle\lhilbert{m}{q_j^*}{q_t}, \cdots \right)$. Secondly, add the first row to the $k+2, \cdots, t$-th row. Thirdly, move the first row to the bottom. Finally delete the first row. It follows that the matrix $R$ in \eqref{eq:matrix7mod8} is equivalent to
\begin{equation}\label{eq:matrix7mod8'}
( \tau, \beta, R^{C}  ).
\end{equation}
where \[\beta := \left( \lhilbert{\sqrt{-m}}{\pi}{\fq_i}\right)^{T}_{2\leq i \leq t}.\]
By Lemma~\ref{lem:matrix} below, $\rank(R)=\rank( \tau, R^{C})$. This proves the case $m\equiv 7\pmod 8$ by Theorem~\ref{thm:4-rankT}.

\medskip

Now we consider the case that $2$ is ramified whence $q_1=2$. Then $m=q_2\cdots q_t\equiv 1\pmod4$ or $m=2q_2\cdots q_t\equiv 2\pmod4$. Write $2\CO_F=\fp^2$. By our condition $2\notin N(F)$, the $B$ in Theorem~\ref{thm:4-rankT} is $B=\{q_0^*:=-1, q_2^*,\cdots, q_t^*\}$.

$(3)$ Suppose $m \equiv 1\pmod8$. Then $A=\{a_0, a_2,\cdots, a_{t-1} \}$. It is clear that the matrix $R$ for $\CT_2(F)$ is
  \[ R= \left(\begin{array}{cccc}
	0   & \cdots & 0  & \cdots    \\
	0  & \cdots & \lhilbert{m}{q_j^*}{q_2}  & \cdots    \\
	\vdots     &   & \vdots &    \\
	0   & \cdots & \lhilbert{m}{q_j^*}{q_{k+1}}+ \lhilbert{m}{q_j^*}{q_{t}}  & \cdots   \\
	\vdots   &   & \vdots &     \\
	0   & \cdots & \lhilbert{m}{q_j^*}{q_{t-1}}+ \lhilbert{m}{q_j^*}{q_{t}} & \cdots \end{array}\right)\]
Firstly, replace the first row $(\cdots)$ by $\left(1,  \cdots,  \scriptstyle\lhilbert{m}{q_j^*}{q_t}, \cdots \right)$. Then we get a matrix whose rank equals $1+\rank~R$. Secondly, add the first row to the $k+2, k+3,  \cdots$ and the $t$-th row. Finally move the first row to the bottom. Now we get $(\tau, R^{C})$. We have $\rk_2~\CT_2(F)=t-2$ by Theorem~\ref{thm:2-rank}.  Thus
\[  \rk_4 \CT_2(-m)= t-2 - \rank~R=t-1 - \rank~(\tau, R^{C}).\]
This proves the case  $m \equiv 1\pmod8$. The arguments for the other cases are similar and we leave the details to the reader. \end{proof}

\begin{lem}\label{lem:matrix}
	Assume that $m\equiv 7\pmod 8$ having a prime factor $q\equiv \pm 3\pmod{8}$. Then $\beta$ is a sum of column vectors of $R^C$.
\end{lem}

\begin{proof}
	Let $\lambda$ be the order of $\fp$ in $\Cl(F)$. Suppose $\pi = \frac{c+d\sqrt{-m}}{2}$ with $c,d\in \BZ$ such that $\pi\CO_F=\fp^\lambda$ and $\hilbert{-1}{\pi}{\fp}=1$. Note that $\lambda$ must be even; otherwise, $2= N( \pi 2^{- \frac{\lambda-1}{2}})\in N(F)$ which contradicts to the assumption.
	 Write $\lambda=2\lambda'$. Then we have a decomposition in $\BZ$
	\begin{equation}\label{eq:pi}
		(2^{\lambda'+1}-c)(2^{\lambda'+1}+c)=md^2.
	\end{equation}
Since $\hilbert{-1}{\pi}{\fp}=1$, it follows from the product formula that $\hilbert{-1}{\pi}{\bar{\fp}}=1$. We obtain $\pi \equiv 1\pmod {\bar{\fp}^2}$ and $\bar{\pi} \equiv 1\pmod {\fp^2}$. But $\fp^\lambda \mid \pi$ and $\lambda$ is even, we have $\pi \equiv 0\pmod {\fp^2}$. Thus
	\[	c= \pi+\bar{\pi}\equiv 1\pmod \fp^2\ \Longrightarrow\  c\equiv 1\pmod 4.\]
Then $2^{\lambda'+1}-c$ and $2^{\lambda'+1}+c$ are coprime, by \eqref{eq:pi}, there exist positive integers $m_{+}, m_{-}, d_{+},d_{-}$ such that $m=m_{+}m_{-}$, $d=d_{+}d_{-}$,  $2^{\lambda'+1}+c = m_{+}d^2_{+},$ and $2^{\lambda'+1}-c = m_{-}d^2_{-}.$ In particular, $m_{+}\equiv c\equiv 1\pmod 4$ and $m_{-}\equiv -1\pmod 4$. We obtain
	\[2c=m_{+}d^2_{+} - m_{-}d^2_{-}. \]
	Now the vector
	\[\beta  = \left( \lhilbert{q_i}{2c}{q_i}  \right)_{1\leq i \leq t}^T.\]
	If $q_i \mid  m_{-}$, noting that $m_{+}\equiv 1\pmod 4$, we have
	\[
	\lhilbert{q_i}{2c}{q_i} = \lhilbert{q_i}{m_{+}}{q_i} =\sum\limits_{q\mid 2m_{+}} \lhilbert{q_i}{m_{+}}{q}= \sum\limits_{q\mid m_{+}}\lhilbert{q_i}{m_{+}}{q}
	=\sum\limits_{q\mid m_{+}}\lhilbert{q_i}{-m}{q}.
	\]
	If $q_i\mid m_{+} $, noting that $-m_{-}\equiv 1\pmod 4$, we also have
	\[
	\lhilbert{q_i}{2c}{q_i}=\lhilbert{q_i}{-m_{-}}{q_i}=\sum\limits_{q\mid 2m_{-}} \lhilbert{q_i}{-m_{-}}{q}=\sum\limits_{q\mid m_{-}}\lhilbert{q_i}{-m_{-}}{q}=\sum\limits_{q\mid m_{-}}\lhilbert{q_i}{-m}{q}=\sum_{q\mid m_{+}}\lhilbert{q_i}{-m}{q}.
	\]
	This means that  $\beta$ is the sum of the column vectors $\left(\lhilbert{q_i}{-m}{q_j}\right)^T_i$ for $q_j\mid m_+$ of $R^C$.
	\end{proof}

\medskip

The rest of this subsection is dedicated to proving Theorem~\ref{thm:4-rank-density}, which is based on the work of Gerth \cite{Ger84} and Yue-Yu \cite{YY04}.  As in the statement of Theorem~\ref{thm:4-rank-density}, $x$ will always denote a positive real number and $t$ will denote a positive integer.

The set $N_{t,x}$ is the disjoint union of subsets $N_{t,x}^{(i)}\ (i=1,2,3)$ defined by (all $p_i$ are odd distinct primes)
\begin{align*}
 &N_{t,x}^{(1)}:=\{m\in N_{t,x}\mid m=p_1\cdots p_t \equiv 3\pmod 4\};\\
 &N_{t,x}^{(2)}:=\{m\in N_{t,x}\mid m=p_1\cdots p_{t-1} \equiv 1\pmod 4\};\\
 &N_{t,x}^{(3)}:=\{m\in N_{t,x}\mid m=2p_1\cdots p_{t-1} \equiv 2\pmod 4\}.
\end{align*}	
Following \cite{Ger84}, we know that when $x\to \infty$,
\[\begin{split}
&\# N_{t;x}^{(1)} \sim \frac{1}{2}\frac{1}{(t-1)!}\frac{x(\log \log x)^{t-1}}{\log x}, \\
& \# N_{t;x}^{(2)} \sim \frac{1}{2}\frac{1}{(t-2)!}\frac{x(\log \log x)^{t-2}}{\log x}=o\left(\# N_{t;x}^{(1)} \right), \\ &\# N_{t;x}^{(3)} \sim \frac{1}{(t-2)!}\frac{x(\log \log (x/2))^{t-2}}{2\log (x/2)}=o\left(\# N_{t;x}^{(1)} \right). \end{split} \]
Here and after we denote $f(x)\sim g(x)$ if $\lim\limits_{x\to \infty} \frac{f(x)}{g(x)}=1$ and $f(x)=o(g(x))$ if  $\lim\limits_{x\to \infty} \frac{f(x)}{g(x)}=0$. Then
\begin{equation}\label{eq:gerth}
	 {\# N_{t;x}}\sim {\# N_{t;x}^{(1)}}\sim \frac{1}{2}\frac{1}{(t-1)!}\frac{x(\log \log x)^{t-1}}{\log x}.
\end{equation}
We define two equivalent relations in $N_{t,x}^{(1)}$.
\begin{defn}
For $m=p_1\cdots p_t, n=q_1\cdots q_t \in N_{t,x}^{(1)}$ arranging such that $p_1<p_2 <\cdots <p_t$ and $q_1<q_2 <\cdots <q_t$, we say that $m$ and $n$ have the same R\'edei type if $q_i\equiv p_i \pmod 4$ for $ i\leq t$ and $\lleg{q_j}{q_i}= \lleg{p_j}{p_i}$ for $1\leq j<i\leq t$; we say that $m$ and  $n$ have the same R\'edei type modulo $8$, if furthermore $q_i\equiv p_i \pmod 8$ for $ i\leq t$. Denote by $[m]$ (resp.  $[[m]]$)  the equivalence class of $m$ with the same R\'edei type (resp.  modulo $8$) respectively.
\end{defn}

\begin{lem}\label{lem:Yueqin}
 For any $m \in N_{t,x}^{(1)}$, we define
  \begin{align*}
     & R(m; t, x):=[m]\cap N_{t,x}^{(1)}=  \{m' \in N_{t,x}^{(1)}\mid m' \text{ and }~m\ \text{have the same R\'edei type}\}, \\
     & S(m; t, x):=[[m]]\cap N_{t,x}^{(1)}=\{m' \in N_{t,x}^{(1)}\mid m' \text{ and }~m\ \text{have the same R\'edei type modulo}~8\}.
  \end{align*}
 Then when $x \rightarrow \infty$, we have
 \begin{equation*}
  \# R(m; t, x) \sim ~ 2^{1-\frac{t^2+t}{2}}\cdot \# N_{t,x}^{(1)},
 \end{equation*}
 and
  \begin{equation*}
\# S(m; t, x) \sim ~  \frac{\# R(m; t, x)}{2^t}.
\end{equation*}
\end{lem}

\begin{proof}
See \cite[Lemma 2.1 and Corollary 2.2]{YY04}.
\end{proof}

\begin{remark}
As mentioned in Gerth \cite[Page 493]{Ger84}, an intuitive explanation of the above lemma might proceed as follows.
To decide the equivalence class $[m]$, we need to fix the conditions  $p_i\pmod 4$ for $l \leq i \leq t-1$ since $m=\prod_{i=1}^{t}p_i\equiv 3\pmod 4$, and the conditions $\lleg{p_j}{p_i}$ for $1\leq j<i\leq t$. Hence, there are $2^{\frac{t^2+t}{2}-1}$ equivalence classes and
the proportion of each equivalence class in $N_{t,x}^{(1)}$ is the same by the above lemma.
Furthermore, given a class $[m]$, then $\{p_1 \pmod 8,\cdots, p_t \pmod 8\}$ have $2^t$ choice. Hence there are $2^t$ modulo $8$ equivalence classes in $[m]$ and the proportion of each modulo $8$ equivalence class in $[m]$ is the same by the above lemma again.
\end{remark}

\begin{lem}\label{lem:2norm_den}
 Let $W(t,x)=\{m\in N_{t,x}^{(1)}\mid 2\in N(\BQ(\sqrt{-m}))\}.$ Then
  \[\lim\limits_{t\rightarrow \infty} \lim\limits_{x\rightarrow \infty} \frac{\#W(t,x)}{\# N_{t,x}^{(1)} }=0.\]
\end{lem}
\begin{proof}
  Put \[f(m)=\begin{cases}
            1, & \mbox{if } 2\in  N(\BQ(\sqrt{-m})) \\
            0, & \mbox{otherwise}.
          \end{cases}\]
Given an equivalence class $[m]$, we claim that there is exactly one class $[[n]]$ in $[m]$ such that $f(n)=1$. Indeed, $q_i\pmod 4$ is determined as $n=q_1\cdots q_t \in [m]$. Then by Hasse's norm theorem, $f(n)=1$ implies that $q_i$ must be $1\pmod 8$ (resp. $7\pmod 8$) in $1\pmod 4$ (resp. $3\pmod 4$). Hence follows the claim.

Now we have
\begin{align*}
   \lim\limits_{t\rightarrow \infty} \lim\limits_{x\rightarrow \infty} \frac{\#W(t,x)}{\# N_{t,x}^{(1)} } & = \lim\limits_{t\rightarrow \infty} \lim\limits_{x\rightarrow \infty} \frac{\sum\limits_{[m]} ~\sum\limits_{[[n]], n\in [m]}  f(n) \cdot \# S(n; t, x)}{\sum\limits_{[m]} ~ \sum\limits_{[[n]], n\in [m] } \# S(n; t, x)}\\
    & =\lim\limits_{t\rightarrow \infty} \lim\limits_{x\rightarrow \infty} \frac{\sum\limits_{[m]} {\# R(m; t, x)}/{2^t}}{\sum\limits_{[m]} \# R(m; t, x)}\\
   \label{eq:2 not Norm} & =\lim\limits_{t\rightarrow \infty} \frac{1}{2^t}=0,
 \end{align*}
where the second equality is by lemma~\ref{lem:Yueqin}.
\end{proof}

\begin{proof}[Proof of Theorem~\ref{thm:4-rank-density}]
By Theorem \ref{thm:rk4-T2(-m)}, Lemma~\ref{lem:2norm_den}, and the estimate \eqref{eq:gerth}, it suffices to prove that for $r\geq 0$,
\[\lim\limits_{t\rightarrow \infty} \lim\limits_{x\rightarrow \infty} \frac{\#\{m\in N_{t,x}^{(1)}\mid\rank~(\tau, R^{C})=t-1-r \}}{\#N_{t;x}^{(1)}}=\frac{\eta_\infty(2)}{2^{r(r+1)}\eta_r(2)\eta_{r+1}(2)}.\]
For any matrix $A\in M_{t-1}(\BF_2)$, write $\Im A:=\{Ax\mid x\in \BF_2^{t-1}\}$. Then we only need to prove that
\begin{align}\label{eq:tau_R1}
  \lim\limits_{t\rightarrow \infty} \lim\limits_{x\rightarrow \infty} \frac{\#\{m\in N_{t,x}^{(1)}\mid \rank~R^{C}=t-1-r, \tau\in \Im R^{C} \}}{\#N_{t;x}^{(1)}}=& \frac{1}{2^r}\cdot \frac{\eta_{\infty}(2)}{2^{r^2}\eta_r(2)^2}
  \end{align}
and
\begin{align}\label{eq:tau_R2}
  \lim\limits_{t\rightarrow \infty} \lim\limits_{x\rightarrow \infty} \frac{\#\{m\in N_{t,x}^{(1)}\mid \rank~R^{\Cl}=t-2-r,\tau \notin \Im R^{\Cl} \}}{\#N_{t;x}^{(1)}}=&\left(1-\frac{1}{2^{r+1}}\right)\cdot \frac{\eta_{\infty}(2)}{2^{(r+1)^2}\eta_{r+1}(2)^2},
\end{align}
since \[\frac{1}{2^r}\cdot \frac{\eta_{\infty}(2)}{2^{r^2}\eta_r(2)^2}+\left(1-\frac{1}{2^{r+1}}\right)\cdot \frac{\eta_{\infty}(2)}{2^{(r+1)^2}\eta_{r+1}(2)^2}= \frac{\eta_\infty(2)}{2^{r(r+1)}\eta_r(2)\eta_{r+1}(2)}.\]

By Lemma \ref{lem:Yueqin} and \cite[Remark 2.3, Equation (3.19)]{YY04}, in each equivalence class $[m]\subset N_{t,x}^{(1)}$, $\tau\in \Im R^{C}$ has probability $\frac{2^{t-1-r}}{2^{t-1}}=\frac{1}{2^r}$ if  $\rank~R^{C}=t-1-r$. i.e.,
\[\lim\limits_{t\rightarrow \infty} \lim\limits_{x\rightarrow \infty} \frac{\#\{m\in N_{t,x}^{(1)}\mid \rank~R^{C}=t-1-r, \tau\in \Im R^{C} \}}{\#\{m\in N_{t,x}^{(1)}\mid \rank~R^{C}=t-1-r\}}= \frac{1}{2^r}. \]
It is proved by Gerth in \cite{Ger84} that
\[\lim\limits_{t\rightarrow \infty} \lim\limits_{x\rightarrow \infty} \frac{\#\{m\in N_{t,x}^{(1)}\mid \rank~R^{C}=t-1-r\}}{\#N_{t;x}^{(1)}}= \frac{\eta_{\infty}(2)}{2^{r^2}\eta_r(2)^2}. \]
This implies the equation \eqref{eq:tau_R1}. The proof of the equation \eqref{eq:tau_R2} is similar and we leave the detail to the reader.
Thus
 \[d^T_{\infty,r}=\frac{1}{2^r}\cdot \frac{\eta_{\infty}(2)}{2^{r^2}\eta_r(2)^2}+\left(1-\frac{1}{2^{r+1}}\right)\cdot \frac{\eta_{\infty}(2)}{2^{(r+1)^2}\eta_{r+1}(2)^2}=\frac{\eta_\infty(2)}{2^{r(r+1)}\eta_r(2)\eta_{r+1}(2)}.\]
 This completes the proof of Theorem~\ref{thm:4-rank-density}. \end{proof}

\section{Study of $\CT_2(\pm l)$ and $\CT_2(\pm 2l)$ for odd prime $l$}\label{sec:special}
By the $2$-rank formula \eqref{eq:2-rankT}, if $l$ is a prime,  $\CT_2(\pm l)$ and $\CT_2(\pm 2l)$ are trivial if $l\equiv \pm 3\pmod{8}$, and  nontrivial cyclic $2$-groups if  $l\equiv \pm 1\pmod 8$. In what follows, we assume $l\equiv \pm 1\pmod{8}$ is a prime. In this section, we shall study the structures of  $\CT_2(\pm l)$ and $\CT_2(\pm 2l)$, or equivalently,  the $2$-power divisibility of their orders   $t_2(\pm l)$ and $t_2(\pm 2l)$.

\subsection{The imaginary case}
\begin{thm}\label{thm:special-family1}
	Let $l\equiv \pm 1 \pmod 8$ be a prime. Then $\CT_2(-l)$ and $\CT_2(-2l)$ are non-trivial cyclic $2$ groups, and
	
$(1)$  $t_2(-l)=2$ if  $l\equiv 7\pmod 8$, $t_2(-l)=4$ if  $l\equiv 9\pmod {16}$, and $t_2(-l)\geq 8$ if $l\equiv 1\pmod {16}$.
	
$(2)$ $t_2(-2l)=2$ if $l\equiv 7\pmod 8$ or $l\equiv 9\pmod {16}$, and $t_2(-2l)\geq 4$ if $l\equiv 1\pmod {16}$.		
\end{thm}

\begin{remark} Based on numerical data, we find out that the conditions $t_2(-l)=2^i$ for $i\geq 3$ and $t_2(-2l)=2^i$ for $i\geq 2$ are not classified by congruence relations.
\end{remark}

\begin{proof} (1) Let $F=\BQ(\sqrt{-l})$. We consider (i)  $l\equiv 7\pmod 8$ and (ii)	
 $l\equiv 1\pmod 8$ separately.

\noindent (i)   In this case $2\nmid h(-l)$ by genus theory. From the commutative diagram \eqref{eq:2.3}, we have
\[
\CT_2(-l)\cong \left(\left(\BZ^\times_2 \times \BZ^\times_2\right)/{\pm 1}\right) [2^\infty]  \cong \BZ/2\BZ.
\]

\noindent (ii)  In this case $2$ is ramified and $F_\fp=\BQ_2(\sqrt{-1})$.

 Let $a= ( \cdots, \underset{\fp}{1+\sqrt{-1}}, \cdots )\in \BA^\times_F$ and $\hat{a}$ be its image in $\CA_F$; here we recall that $\CA_F$ is the group defined in \eqref{eq:2.1} with $p=2$. Then $a^4= ( \cdots, \underset{\fp}{-4}, \cdots ) \in \overline{F^\times \prod_{v\notin S} U_v} $ and hence $\hat{a}\in \CT_2(F)[4]$. Since $a^2= (\cdots, \underset{\fp}{2\sqrt{-1}}, \cdots  ) \equiv (\cdots, \underset{\fp}{\sqrt{-1}}, \cdots  ) \bmod \left(\overline{F^\times \prod_{v\notin S} U_v}\right)$ and$\sqrt{-1}$ is nontrivial in $ U_{1,\fp}/\{\pm1\} \subset {\CA}^{\mathrm{pro-}2}_F$, we have $\hat{a}^2 \neq 0$ in $\CT_2(F)$. Thus $\hat{a}$ is a generator of the cyclic group $\CT_2(F)[4]$.

The $2$-units $E_S$ of $F$ is generated by $-1$ and $2$. Clearly, $2\notin N( E_S)$. Write $l=2g^2-h^2$. Let $\alpha =h+\sqrt{-l}$. By Proposition~\ref{prop:selmer}, $L=F\left(\sqrt{-1},\sqrt{l},\sqrt{2},\sqrt{\alpha}\right)=F\left(\sqrt{-1},\sqrt{2},\sqrt{\alpha}\right)$ is the maximal abelian extension of exponent $2$ over $F$ unramified outside $2$. The map
 \[ \CT_2(F)[4]\rightarrow \Gal(F(\sqrt{-1})/F)\times \Gal(F(\sqrt{2})/F)\times \Gal(F(\sqrt{\alpha})/F)   \]
has kernel $2\CT_2(F)[8]$. Thus $t_2(-l)\geq 8$ if and only if the additive Artin symbols $[a, -1]=[a, 2]=[a,\alpha]=0$. It is easy to see $[a, 2]=[a,-1]=0$ since $l\equiv 1\pmod 8$.
We have
\[\begin{split}
[a, h+\sqrt{-l}]= & \lhilbert{1+\sqrt{-1}}{h+\sqrt{-l}}{F_\fp}  = \lhilbert{-\sqrt{l}}{h+\sqrt{-l}}{F_\fp}+ \lhilbert{-\sqrt{l}-\sqrt{-l}}{h+\sqrt{-l}}{F_\fp} \\
= &  \lhilbert{-\sqrt{l}}{2g^2}{\BQ_2} + \lhilbert{-\sqrt{l}-\sqrt{-l}}{h+\sqrt{-l}}{F_\fp}\\
=  &  \lhilbert{\sqrt{l}}{2}{\BQ_2} +\lhilbert{-\sqrt{l}-\sqrt{-l}}{h+\sqrt{-l}}{F_\fp}.\\
\end{split}
\]
Note that
\[
\lhilbert{\sqrt{l}}{2}{\BQ_2}=
\begin{cases}
0 & \text{ if } l\equiv 1\pmod {16},\\
1 & \text{ if } l\equiv 9\pmod {16}.
\end{cases}
\]
For any $x,y\in F_\fp$, noting that $-1$ is a square in $F_\fp$, we have  \[0=\lhilbert{\frac{x}{x+y}}{\frac{y}{x+y}}{F_\fp}= \lhilbert{xy}{x+y}{F_\fp}+ \lhilbert{x+y}{x+y}{F_\fp}+\lhilbert{x}{y}{F_\fp}.\] Put $x= -\sqrt{l}-\sqrt{-l}, y= h+\sqrt{-l}$. Note that $x+y\in \BQ_2$. It follows that $\lhilbert{x+y}{x+y}{F_\fp}=0$.
Thus,
\[
\lhilbert{x}{y}{F_\fp}=\lhilbert{x+y}{xy}{F_\fp}=\lhilbert{x+y}{4g^2 l }{\BQ_2}=0.
\]
This proves (1).

(2) follows from the same argument used in the proof of (1). We omit the details here.
\end{proof}

\subsection{The real case}
We need the following formula of Coates (see \cite[Appendix]{Coa77} or \cite[Chapter III. 2.6.5]{Gra03}):
\begin{prop}\label{prop:coates}
	Let  $K\neq \BQ$  be a totally real number field. Assume that the Leopoldt Conjecture holds for $(p,K)$, i.e., $\delta_p(K)=0$. Then
	\begin{equation}\label{eq:coates}
	\# \CT_{p}(K)=(p \text{-adic unit})\cdot \frac{p\cdot[K \cap \BQ^{p,\cyc} : \BQ]\cdot h(K) \cdot R_p(K)}{\sqrt {D_K} \cdot \prod\limits_{\mathfrak{p}|p} N\mathfrak{p}}. \end{equation}
	Here $h(K)$ is the class number,  $R_p(K)$ is the  $p$-adic regulator and $D_K$ is the discriminant of $K$, $\BQ^{p,\cyc}$ is the cyclotomic $\BZ_p$-extension of $\BQ$, and the product runs over all primes of $K$ lying above $p$ and $N$ is the norm map from $K$ to $\BQ$.
\end{prop}

\begin{lem}\label{lem:units}
Assume $l \equiv \pm1 \pmod8$ is a prime. Let $\nu_2$ be the normalized $2$-adic valuation and $\log_2$ be the $2$-adic logarithmic map. For $m=l$ or $2l$, let $\varepsilon_m=a_m+b_m\sqrt{m}$ be the fundamental unit of $\BQ(\sqrt{m})$. Then
\begin{enumerate}
	\item $\nu_2(t_{2}(l))= \nu_2(\log_2(\varepsilon_l) ) -1 = \nu_2(a_l)-1$.
	\item $\nu_2(t_2(2l))=\nu_2(h(2l))+\nu_2(b_{2l})-1.$
\end{enumerate}
\end{lem}

\begin{proof} (1) Let $F=\BQ(\sqrt{l})$. Recall that the $2$-adic regulator  $R_2(F)$ is $\log_2(\varepsilon_l)$. By Coates' formula above, we have $\nu_2(t_{2}(l))=\nu_2(\log_2(\varepsilon_l))-1$ as $2\nmid h(l)$. It remains to show that $\nu_2(\log_2(\varepsilon_l))=\nu_2(a_l)$. We shall use the basic property of logarithm that, for $x\in \overline{\BQ}_2$, if $\nu_2(x-1)>1$ then $ \nu_2( \log_2(x) )=\nu_2(x)$.

If $l\equiv 1\pmod 8$,  then it is easy to see $a_l$ and $b_l$ are integers. It is also known that    $N(\varepsilon_l)=a^2_l- l b^2_l=-1$. It follows that  $4\mid a_l$ and $b_l$ is odd. Thus $\nu_2(\varepsilon^2_l-1)=\nu_2(\varepsilon^2_l+\varepsilon_l\bar{\varepsilon}_l)=1+\nu_2(a_l)\geq 3$. This implies that $\nu_2(\log_2(\varepsilon^2_l))=\nu_2(\varepsilon^2_l-1)=1+\nu_2(a_l)$. Hence, $\nu_2(\log_2(\varepsilon_l) )=\nu_2(a_l)$.

If $l\equiv 7\pmod 8$, we first prove that $a_l$ is even.
In this case $2$ is ramified in $F$, say $2\CO_F=\fp^2$. Since $h(l)$ is odd, $\fp$ must be principal, say $\fp=(\pi)$ with $\pi \in \CO_F$. Then $\pi^2/2$ is a unit, say $\varepsilon_l
^k$. Note that $k$ must be odd. Otherwise, $\sqrt{2}\in F$, which is absurd. Then $(\pi \varepsilon_l^{-(k-1)/2})^2=2\varepsilon_l$ and hence $\pi \varepsilon_l^{-(k-1)/2}\in \CO_F$.  Write $\pi \varepsilon_l^{-(k-1)/2}=c+d\sqrt{l}$ with $c,d\in \BZ$. Then $c$ and $d$ must be odd since $N(c+d\sqrt{l})=2$. Hence $a_l=\frac{c^2+d^2l}{2}$ is clearly even.

Thus,  $b_l$ must be odd. Then $ \nu_2(\varepsilon^4_l-1)=\nu_2(\varepsilon_l^4-\varepsilon_l^2 \bar{\varepsilon}_l^2)=2+\nu_2(a_lb_l)=2+\nu(a_l)$.  Therefore, $\nu_2(\log_2(\varepsilon_l))=\nu_2(a_l)$. This completes the proof of (1).

(2) 	Clearly $a_{2l}$ is odd and $b_{2l}$ is even. 		
We have
 \[ \nu_2(\varepsilon^4_{2l}-1) = \nu_2(\varepsilon_{2l}^2+\varepsilon_{2l}\bar{\varepsilon}_{2l})+ \nu_2(\varepsilon_{2l}^2-\varepsilon_{2l}\bar{\varepsilon}_{2l}) = \nu_2(2a_{2l} )+\nu_2(2\sqrt{2l}b_{2l})=\frac{5}{2}+\nu_2(b_{2l}). \]
Hence, $\nu_2(\log_2(\varepsilon_{2l}))= \frac{1}{2}+\nu_2(b_{2l})$. Then (2) follows from Coates' formula for $\CT_2(\BQ(\sqrt{2l}))$.
\end{proof}
\begin{remark} The proof of $a_l$ is even for $l\equiv 7\pmod 8 $ holds for $l\equiv 3\pmod{4}$. For a different proof of this fact, see \cite{ZY14}.
\end{remark}

 The following proposition collects results about the $2$-class groups $\Cl_2(-l)$ and $\Cl_2(-2l)$  due to Gauss, Hasse~\cite{Has69}, Brown~\cite{Brown72} and others, most importantly due to Leonard-Williams~\cite{LW82}, see \cite[Theorem 4.2]{LX20} for a proof about  $\Cl_2(-2l)$.
\begin{prop} \label{prop:l78} Let $l$ be an odd prime. Then both $\Cl_2(-l)$ and $\Cl_2(-2l)$ are cyclic groups.
	
$(1)$ $h_2(-l)=1$ if $l\equiv 3\pmod 4$,
	 $h_2(-l)=2$ if $l\equiv 5\pmod 8$ and $h_2(-l)\geq 4$ if $l\equiv 1\pmod 8$. Moreover, if $l\equiv 1\pmod{8}$, suppose $l=2g^2-h^2$, then $h_2(-l)=4$ if and only if $g\equiv 3\pmod{4}$,  $h_2(-l)=8$ if and only if
	 $\Bigl(\frac{2h}{g}\Bigr) \Bigl(\frac{g}{l}\Bigr)_4=-1$.

$(2)$ 	 $h_2(-2l)=2$ if $l\equiv \pm 3\pmod 8$  and $h_2(-2l)\geq 4$ if $l\equiv \pm 1\pmod 8$. Moreover,
\begin{enumerate}
	\item[(i)] If $l\equiv 1\pmod{8}$, suppose $l=u^2-2v^2$ such that $u\equiv 1\pmod{4}$, then $h_2(-2l)=4$ if and only if $u\equiv 5\pmod{8}$, $h_2(-2l)=8$ if and only if $\Bigl(\frac{u}{l}\Bigr)_4=-1$.

	\item[(ii)] If $l\equiv 7\pmod{8}$, then  $h_2(-2l)=4$ if and only if $l\equiv 7\pmod{16}$, and $h_2(-2l)=8$ if and only if $l\equiv 15\pmod{16}$ and $(-1)^{\frac{l+1}{16}} \left(\frac{2u}{v}\right)=-1$ where $(u,v)\in \BZ_{>0}^2$ satisfying $l=u^2-2v^2$.
	\end{enumerate}	
\end{prop}
We have the following theorem:

\begin{thm} \label{thm:l78}
	Assume $l\equiv 7\pmod 8$ is a prime.  Then $\CT_2(l)$ and $\CT_2(2l)$ are non-trivial $2$-cyclic groups, 	$4\mid t_2(l)$ and
		\begin{enumerate}
		\item	$t_2(l)=4\Leftrightarrow t_2(2l)=2 \Leftrightarrow h_2(-2l)=4\Leftrightarrow  l\equiv 7\pmod{16}$;
	\item   	$t_2(l)=8\Leftrightarrow t_2(2l)=4 \Leftrightarrow h_2(-2l)=8\Leftrightarrow\ l\equiv 15\pmod{16}$ and $(-1)^{\frac{l+1}{16}} \left(\frac{2u}{v}\right)=-1$ where $(u,v)\in \BZ_{>0}^2$ is a solution of $l=X^2-2Y^2$.
		\end{enumerate}	
Consequently, we always have $t_2(l)\equiv 2t_2(2l) \equiv h_2(-2l)\pmod{16}$.
\end{thm}
\begin{remark}
	However, in general the three numbers $t_2(l)$, $2t_2(2l)$ and $h_2(-2l)$ are not equal if one (hence all) of them $\geq 16$. For example, let $l=223$, then $t_2(l)=16$, $2t_2(2l)=256$ and  $h_2(-2l)=32$.
\end{remark}
\begin{proof} (1) We first study $t_2(l)$.
As shown in the proof of Lemma~\ref{lem:units}, $\varepsilon_l
= a_l+b_l\sqrt{l}= \frac{1}{2}(c+d\sqrt{l})^2$ where $c,d$ are odd integers and $N(c+d\sqrt{l})=c^2-d^2l=2$. In particular, $c^2\equiv 2\pmod d$. It follows that every prime factor of $d$ is congruent to $\pm 1\pmod 8$. Hence $d^2\equiv 1\pmod {16}$ and $\nu_2(a_l) = \nu_2(1+d^2 l )$. For $l\equiv 7\pmod{8}$, $ \nu_2(1+d^2 l)\geq 3$, with equality if and only if $l\equiv 7\pmod{16}$. By Lemma~\ref{lem:units}(1), $4\mid
t_2(l)=2^{\nu_2(ld^2+1)-1}$, and $t_2(l)=4$ if and only if  $l\equiv 7\pmod {16}$.

Note that the Jacobi symbol $\leg{2u}{v}$ is independent on the choices of $u$ and $v$ (see \cite[Lemma 4.1]{LX20}). By the results of Leonard-Williams (Proposition~\ref{prop:l78}(2)), we are left to show that if $l\equiv 15\pmod{16}$, then
\[ \nu_2(ld^2+1)=4  \Longleftrightarrow \ (-1)^{\frac{l+1}{16}}\leg{2u}{v}=-1. \]

Since $l=(u+ \sqrt{2}v)(u-\sqrt{2}v)\mid ld^2=(c+\sqrt{2})(c-\sqrt{2})$, one of the prime elements $u\pm \sqrt{2}v$ must divides $c+\sqrt{2}$ in the Euclidean domain  $\BZ[\sqrt{2}]$.

(i) Suppose  $\frac{c+\sqrt{2}}{u+\sqrt{2}v}\in \BZ[\sqrt{2}]$. Note that $c+ \sqrt{2}$ and $c- \sqrt{2}$ are coprime in $\BZ[\sqrt{2}]$, the integers  $\frac{c+\sqrt{2}}{u+\sqrt{2}v}$ and $\frac{c-\sqrt{2}}{u-\sqrt{2}v}$ are coprime, but their product is $d^2$ and  $\BZ[\sqrt{2}]$ has class number $1$,  hence there exist $s,t\in \BZ$  and $\varepsilon \in \{1, 1+\sqrt{2}\}$ such that
\[\frac{c+\sqrt{2}}{u+\sqrt{2}v}=\varepsilon (t-s\sqrt{2})^2. \]
Since the left hand side is totally positive, we must have $\varepsilon=1$.  Comparing the coefficients of $\sqrt{2}$ gives \begin{equation}\label{key}
1=(t^2+2s^2)v-2tsu.
\end{equation}
Note that $ts$ must be positive. We  may assume that $t,s$ are both positive. Since $l=u^2-2v^2\equiv -1\pmod {16}$,  both $u$ and $v$ are odd. In fact, $v\equiv 1\pmod 4$ by \eqref{key}. Hence
$\leg{2u}{v}=\leg{-st}{v} =\leg{t}{v}  \leg{s}{v}$. Note that $d,t$ are odd. By quadratic reciprocity law, $\leg{t}{v}=\leg{v}{t}=\leg{2}{t}$. The last equality follows from \eqref{key}.  Write $s=2^r s_0$ with $2\nmid s_0$. If $s \equiv 2\pmod4$,  then $v \equiv 5 \pmod8$ and  $\leg{s}{v}=\leg{2\cdot s_0}{v}=-\leg{v}{s_0}$.  If $s \equiv 0 \pmod4$, then $t^2\equiv 1\pmod8$ and $v \equiv 1 \pmod8$. So $\leg{s}{v}=\leg{s_0}{v}=\leg{v}{s_0}=1$. If
$s \equiv \pm 1\pmod4$, then $\leg{s}{v}=\leg{v}{s}=1$. Hence
\[\leg{s}{v}=\begin{cases}
-1, & \mbox{if } s \equiv 2 \pmod4, \\
1, & \mbox{otherwise}.
\end{cases}\]
Therefore  $\leg{2u}{v}=1$  if and only if $\pm d=t^2-2s^2 \equiv \pm1 \pmod {16}$. This implies that  $16 \parallel ld^2+1$ if and only if $ (-1)^{\frac{l+1}{16}}\leg{2u}{v}=-1$.

(ii) Suppose  $\frac{c+\sqrt{2}}{u-\sqrt{2}v}\in \BZ[\sqrt{2}]$. By similar argument,  there exist two positive integers $t,s$ such that
\[1=2stu-(t^2+2s^2)v. \]
For this equation,  $v\equiv 3\pmod 4$ and  $\leg{2u}{v}=\leg{t}{v} \leg{s}{v}$. One can repeat the argument above to obtain that  $\leg{t}{v}=\leg{2}{t}$ and that  \[\leg{s}{v}=\begin{cases}
-1, & \mbox{if } s \equiv 2 \pmod4, \\
1, & \mbox{otherwise}.
\end{cases}\]
Again this implies that  $16 \parallel ld^2+1$ if and only if $ (-1)^{\frac{l+1}{16}}\leg{2u}{v}=-1$.

(2) If $l\equiv 7\pmod8$, then $h(2l)$ is odd. By Lemma~\ref{lem:units}(2), $\nu_2(t_2(2l))=\nu_2(b_{2l})-1$. According to the last paragraph in \cite[\S~3]{LW82}, we have
$h(-2l)\equiv b_{2l} \pmod {16}$. Then $t_2(2l)=\frac{h_2(-2l)}{2}=\frac{t_2(l)}{2}$ if $t_{2l}=2$ or $4$. We just need to apply Proposition~\ref{prop:l78}.
\end{proof}

\begin{prop}\label{prop:special-fam}
Assume $l\equiv 1\pmod 8$ is a prime.

$(1)$ Write $l=2g^2-h^2$ with $g,h\in \BZ_{>0}$. Then
 \begin{align}  \label{eq:special1} &t_2(l)=2 \Longleftrightarrow h_2(-l)=4  \Longleftrightarrow g\equiv 3\pmod{4}; \\ \label{eq:special2}
	& t_2(l)=4  \Longleftrightarrow \begin{cases}
		h_2(-l)=8  & \mbox{if } l \equiv 1 \pmod {16} \\
		h_2(-l)\geq 16  & \mbox{if } l \equiv 9 \pmod {16}
	\end{cases}\ \Longleftrightarrow (-1)^{\frac{l-1}{8}} \leg{2h}{g}\leg{g}{l}_4=-1.
\end{align}

$(2)$ Write $l=u^2-2v^2$ with $u, v\in \BZ_{>0}$ and $u \equiv 1 \pmod4$. Then
 \begin{align}  \label{eq:special3} &t_2(2l)=2\Longleftrightarrow \leg{u}{l}=-1,\\
 	 \label{eq:special4} &t_2(2l)=4\Longleftrightarrow (-1)^{\frac{l-1}{8}}\leg{u}{l}_4=-1.
 	 \end{align}
\end{prop}
\begin{proof} (1) For $l\equiv 1\pmod 8$, Williams \cite{Wil81} proved that
		\[	a_l\equiv	\begin{cases}  h(-l)+l-1 \pmod {16},  & \text{ if } h_2(-l)\geq 8; \\   4(h(l)-1)+l-1-h(-l) \pmod {16}, & \text{ if } h_2(-l)=4.\end{cases} \]
Hence	we have
	\begin{equation}
		\begin{cases}2t_2(l) \equiv  h(-l)+l-1 \pmod {16}, & \text{ if } h_2(-l)\geq 8;\\
			t_2(l)=2, & \text{ if } h_2(-l)=4. \end{cases}
	\end{equation}
	Applying Coates' formula \eqref{eq:coates}, Lemma~\ref{lem:units} and Proposition~\ref{prop:l78}(1), we get the result.

(2) It follows from \eqref{eq:coates} that
$t_2(2l)$ is equal to $\log_2(\varepsilon_{2l}) h(2l)/{(2\sqrt{2})}$ up to a $2$-adic unit. Denote by $R+S\sqrt{2l}$ the fundamental unit of norm $1$ of $\BQ(\sqrt{2l})$ and by $h^{+}(2l)$ the narrow class number of $\BQ(\sqrt{2l})$. Then, $R+S\sqrt{2l}=\varepsilon_{2l}$ and $h^{+}(2l)=2h(2l)$ if $N(\varepsilon_{2l})=1$; $R+S\sqrt{2l}=\varepsilon_{2l}^2$ and $h^{+}(2l)=h(2l)$ if $N(\varepsilon_{2l})=-1$. Thus, by Lemma~\ref{lem:units}(2), we have
	\[\nu_2(t_2(2l))=\nu_2(h^{+}(2l))+\nu_2(S)-2. \]
	The main theorem in \cite{KW82} tells us that
	\[\frac{S\cdot h^{+}(2l)}{2} \equiv 1-l-h(-2l) \pmod {16}. \]
	Then all results here directly follow the discussion in \cite[\S~2]{LW82}.
\end{proof}

Now we can prove the density result about $\CT_2(l)$ and $\CT_2(2l)$:

\begin{proof}[Proof of Theorem~\ref{thm:i=0,1}]
  (1) We first show \eqref{eq:densityaab}. In the case $e=0$, then $l\equiv 1\pmod 8$.   Stevenhagen \cite[Theorem 1]{Ste93} proved that $h_2(-l)\geq 8$ if and only if $l$ splits completely in $\BQ(\zeta_8,\sqrt{1+i})$. Then by Chebotarev's density theorem,
   \[ \lim\limits_{x \rightarrow \infty}\frac{\# \{ l \leq x:  l \equiv 1 \pmod{8},\  h_2(-l)= 4\}}{\#\{\ l \leq x: l \equiv 1\pmod{8}\}}= \lim\limits_{x \rightarrow \infty}\frac{\# \{ l \leq x:  l \equiv 1 \pmod{8},\  h_2(-l)\geq 8\}}{\#\{\ l \leq x: l \equiv 1 \pmod{8}\}}=
  \frac{1}{2}. \]
  By \eqref{eq:special1} in Proposition~\ref{prop:special-fam}, the case $i=0$ follows.

  Recently, Koymans (\cite[Theorem 1.1]{Koy20}) proved that
  \[\lim\limits_{x \rightarrow \infty}\frac{\# \{ l \leq x:  l \equiv 1 \pmod8  \text{ and }  h_2(-l)=8 \}}{\#\{\ l \leq x: l \equiv 1 \pmod8\}}=\frac{1}{4}.\]
  As a corollary of \cite[Theorem 1]{Ste93}, we have that $l\equiv 9\pmod{16}$ such that $h_2(-l)\geq 8$ if and only if the Frobenius of $l$ in $\Gal(\BQ(\zeta_{16},\sqrt{1+i})/\BQ$ acts trivially in $\BQ(\zeta_8,\sqrt{1+i})$ and  maps $\zeta_{16}$ to $-\zeta_{16}$. Hence
   \[ \lim\limits_{x \rightarrow \infty}\frac{\# \{ l \leq x:  l \equiv 9 \pmod{16},\  h_2(-l)\geq 8\}}{\#\{\ l \leq x: l \equiv 9 \pmod {16}\}}= \lim\limits_{x \rightarrow \infty}\frac{\# \{ l \leq x:  l \equiv 1 \pmod{16},\  h_2(-l)\geq 8\}}{\#\{\ l \leq x: l \equiv 1 \pmod {16}\}}=
  \frac{1}{2}. \]
 If we can show
  \begin{equation} \label{eq:l18density}
  	\lim\limits_{x \rightarrow \infty}\frac{\# \{ l \leq x:  l \equiv 9 \pmod{16},\    h_2(-l)=8\}}{\#\{\ l \leq x: l \equiv 9 \pmod{16}\}} =\lim\limits_{x \rightarrow \infty}\frac{\# \{ l \leq x:  l \equiv 9 \pmod{16},\  h_2(-l)\geq 16\}}{\#\{\ l \leq x: l \equiv 9 \pmod {16}\}}=
  	\frac{1}{4},
  \end{equation}
  then
  \[   \lim\limits_{x \rightarrow \infty}\frac{\# \{ l \leq x:  l \equiv 1 \pmod{16}  \text{ and }  h_2(-l)\geq 16 \}}{\#\{\ l \leq x: l \equiv 1 \pmod{16}\}}= \lim\limits_{x \rightarrow \infty}\frac{\# \{ l \leq x:  l \equiv 1 \pmod{16}  \text{ and }  h_2(-l)=8 \}}{\#\{\ l \leq x: l \equiv 1 \pmod{16}\}}=  \frac{1}{4}. \]
Hence the $i=1$  case follows from \eqref{eq:special2}. It suffices to show \eqref{eq:l18density}.

Let
 \[  e_l = \begin{cases}
1,   & \text{ if } h_2(-l)\geq 16,\\
-1,  &  \text{ if } h_2(-l)=8,\\
0,  & \text{ if } h_2(-l)=4.
\end{cases}   \]
By\cite[Theorem 1.2]{Koy20}, we have
  \begin{equation} \label{eq: Koy1}\sum_{l\leq x,\ l\equiv 1\pmod8}e_l \ll  x/\exp((\log x)^{0.1}). \end{equation}
Replacing the spin symbol $[w]$ in \cite[Lemma 4.1, 4.2]{Koy20} by the twisted symbol $[w]':=[w]\cdot \lambda(w)$ for all totally positive elements $w$ of $\BZ[\zeta_8]$, where  $\lambda(w)=(-1)^{\frac{Nw-1}{8}}$ if $Nw \equiv 1\pmod8$ and 1 otherwise,  one    follows the argument  there and obtains
  \begin{equation}\label{eq: Koy2}
  \sum_{l\leq x,\ l\equiv 1\pmod8}(-1)^{\frac{l-1}{8}}e_l \ll  x/\exp((\log x)^{0.1}). \end{equation}
Thus
  \[\sum_{l\leq x,\, l\equiv 1\pmod8}\left(e_l-(-1)^{\frac{l-1}{8}}e_l\right) =2\sum_{l\leq x,\, l\equiv 9\pmod {16}}\left(1_{16|h(-l)}- 1_{8||h(-l)}\right)\ll  x/\exp((\log x)^{0.1}).\]
Note that as $x\rightarrow +\infty$, $\log x=o(\exp((\log x)^{0.1})$, by Dirichlet's density theorem, then
 \[  \#\{l\leq x,\ l\equiv 9\pmod{16},\ h_2(-l)=8 \}\sim \#\{l\leq x,\ l\equiv 9\pmod{16},\ h_2(-l)\geq 16\}\sim \frac{x}{32\log x}. \]
Hence we  have   \eqref{eq:l18density}.

  In the case $e=1$, $l\equiv 7\pmod 8$.  By Theorem~\ref{thm:l78}, the case $i=0$ follows from the fact that $t_2(l)=4$ if and only if  $l\equiv 7\pmod{16}$, and the case $i=1$ follows from the following result of Milovic \cite[Theorem 1]{Mil17} on $h_2(-2l)$ that
  \[\lim\limits_{x \rightarrow \infty}\frac{\# \{\ l \leq x: l \equiv -1 \pmod8,\ h_2(-2l)=8 \}}{\#\{\ l \leq x: l \equiv -1 \pmod8\}}=\frac{1}{4}. \]

(2) Case $7\pmod 8 $ for \eqref{eq:density2aaa} follows from (1) and Theorem~\ref{thm:l78}, and case $1\pmod{8}$ follows from Proposition~\ref{prop:special-fam}(2) and \cite[Theorem 1]{KM19} with similar arguments for $t_2(l)$; we omit the details.
\end{proof}

\begin{remk}\label{rmk:finer} We actually proved that for $i=1$ and $2$,
\[	\lim\limits_{x \rightarrow \infty}\frac{\#\{l \leq x: l \equiv 1 \pmod {16},\ t_2(l)= 2^{i}\}}{\#\{l \leq x: l \equiv 1 \pmod {16}\}}=
				\lim\limits_{x \rightarrow \infty}\frac{\#\{l \leq x: l \equiv 9 \pmod {16},\ t_2(l)= 2^{i}\}}{\#\{l \leq x: l \equiv 9 \pmod {16}\}}= \frac{1}{2^{i}}.
\]
\end{remk}

\section{Distribution Conjectures for $\CT_p$-groups of quadratic fields}

\subsection{Distribution conjecture of $\CT_p$ in the full family} We first propose a distribution conjecture on the group structure of $\CT_p(F)$ when $F$ varies in the family of all imaginary (resp. real) quadratic fields $\CF_{\mathrm{im}}$ (resp. $\CF_{\mathrm{re}}$).

For $5\leq p\leq 47$, numerical data presented in \cite[\S 5.2]{PV15} suggested that of all real quadratic fields $\BQ(\sqrt{m})$ such that  $m\leq 10^9$ is squarefree, the proportion of fields with trivial $\CT_p$-groups (so-called $p$-rational field) is close to $\eta_\infty(p)$. It was mentioned there that the authors also considered the distribution about the group structures of $\CT_p$-groups, however, we did not find any further statement and subsequent studies  in the literature.

Based on Theorem~\ref{thm:4-rank-density} and numerical data in the appendix, we propose the following conjecture:

\begin{conj}\label{conj:CL-full-family}
	Let $p$ be a prime. Let $\CF_{\mathrm{im}}$ (resp. $\CF_{\mathrm{re}}$) be the family of all imaginary (resp. real) quadratic fields. For each finite abelian $p$-group $G$, one has
	\begin{align} \label{eq:conj11}
		&\lim\limits_{x\rightarrow \infty} \frac{ \# \{ F\in \CF_{\mathrm{im}}\mid  -D_F \leq x,\ 6\CT_p(F) \cong G\}}{{\#\{ F\in \CF_{\mathrm{im}}:   -D_F \leq x\}}} = \frac{\eta_{\infty}(p)/ \eta_1(p)}{\#G  \cdot\#\Aut(G) };\\
\label{eq:conj12}
		& \lim\limits_{x\rightarrow \infty} \frac{ \# \{ F\in \CF_{\mathrm{re}}\mid  D_F \leq x,\ 6\CT_p(F) \cong G\}}{{\#\{ F\in \CF_{\mathrm{re}}:   D_F \leq x\}}} = \frac{\eta_{\infty}(p)}{ \#\Aut(G) }.
	\end{align}
Here $D_F$ is the discriminant of $F$, and we recall that $\eta_s(q):=\prod\limits_{i=1}^{s}(1-q^{-i})$ for  $s\in \BZ_{>0}\cup \{\infty\}$ and $q>1$.
\end{conj}
\begin{remark}
(1) For $p\geq 5$, we have $6\CT_p(F) \cong \CT_p(F)$ and hence the factor $6$ can be removed from the statement of our conjecture. For $p=2$ or $3$, we have $6\CT_p(F) = p\CT_p(F)$. For $p=5$ and $7$, we have carried out numerical computation of $\CT_p(F)$ with $|D_F|\leq 5\times 10^7$; see Tables~\ref{table:t5real} - \ref{table:t7ima}, which give strong evidence of Conjecture~\ref{conj:CL-full-family} in these cases.
	
(2) In the bad primes $2$ and $3$ case, when the bound is $5\times 10^7$, the distributions of $2\CT_2$ and $3\CT_3$ are actually not quite good based on our computation, but this is expected just like the analogue phenomenon for the distributions of narrow $2$-class groups and tame kernels of quadratic fields: the bound is not big enough. We gain confidence from recent breakthrough of Smith\cite{Smi17} on the distribution of narrow $2$-class groups of quadratic fields, as well as the $4$-rank density formula for $\CT_2$ of imaginary quadratic fields we just proved here.

(3) If using the setting of local Cohen-Lenstra Heuristic,   the weight function for $p$-class groups is $\omega_0$ for imaginary quadratic fields and $\omega_1$ for real ones where
 \begin{equation}\label{eq:wt}
  \omega_i(G)=\frac{1}{(\# G)^{i}\cdot \#\Aut(G)},
 \end{equation}
 the weight functions for $\CT_p$-groups are exactly the reverse order.

(4) For more general conjectures on distributions of $\CT_p$-groups of quadratic fields, which are also in the spirit of the Cohen-Lenstra heuristics, see \cite{LOX21}.
\end{remark}

\subsection{Distribution conjecture of $\CT_2$ in sub-families}

\begin{conj}\label{conj:density1}
	Assume all $l$ appeared below are primes. For each integer $i\geq 0$ and $e\in \{0,1\}$, then
	\begin{align} \label{eq:conj21}
		&\lim\limits_{x \rightarrow \infty}\frac{\#\{l \leq x: l \equiv 1 \pmod {16}, t_2(-l)=2^{i+3}\}}{\#\{l \leq x: l \equiv 1 \pmod {16}\}}=\frac{3}{4^{i+1}}.
\\ \label{eq:conj21.5}
&\lim\limits_{x \rightarrow \infty}\frac{\#\{l \leq x: l \equiv 1 \pmod {16}, t_2(-2l)=2^{i+2}\}}{\#\{l \leq x: l \equiv 1 \pmod {16}\}}=\frac{3}{4^{i+1}}.		
\\ \label{eq:conj22}
	&\lim\limits_{x \rightarrow \infty}\frac{\#\{l \leq x: l \equiv (-1)^e \pmod 8,\ t_2(l)= 2^{i+1+e}\}}{\#\{l \leq x: l \equiv (-1)^e \pmod 8\}}=\frac{1}{2^{i+1}}.
\\ \label{eq:conj23}
	&\lim\limits_{x \rightarrow \infty}\frac{\#\{l \leq x: l \equiv (-1)^e \pmod 8,\ t_2(2l)=2^{i+1}\}}{\#\{l \leq x: l \equiv (-1)^e \pmod 8\}}=\frac{1}{2^{i+1}}.
	\end{align}
\end{conj}
We shall present numerical evidence  in Tables~\ref{table:-l116} - \ref{table:2l78}.

\begin{remark}
	(1) Under the setting of extended local Cohen-Lenstra heuristic, one can interpret \eqref{eq:conj21} more conceptually as follows. Let $\CM_k=\{ \BZ/{2^{i+k}\BZ}\mid i\geq 0 \}$ for $k\geq 1$. For $G=\BZ/{2^{i+k}\BZ}\in \CM_k$,  then a direct computation gives
	\begin{equation*}
		\frac{\omega_1(G )}{\sum_{H\in \CM_k} \omega_1( H)} =\frac{3}{4^{i+1}}.
	\end{equation*}
	Thus  \eqref{eq:conj21}  is equivalent to that the natural density of primes $l$ with $\CT_2(-l)\cong G$ among all primes $\equiv 1\pmod {16}$ is equal to the ratio of $\omega_1(G)$ to the total $1$-weight of the space $\CM_3$. For \eqref{eq:conj21.5},  the corresponding space is $\CM_2$.

	(2) One can also reformulate \eqref{eq:conj22} and \eqref{eq:conj23} by using the weight function $\omega_0$ and by noting the following identity:
	\begin{equation*}
	\frac{\omega_0(G)}{\sum_{H\in \CM_k} \omega_0( H)} = \frac{1}{2^{i+1}}, \quad \text{ where } \quad G = \BZ/2^{k+i}\BZ.
	\end{equation*}
In \eqref{eq:conj22} (resp. \eqref{eq:conj23}), the total space is $\CM_{e+1}$	(resp. $\CM_{1}$).
	%, using instead the weight function $\omega_0$ as follows.

	(3) By  Lemma~\ref{lem:units}(1), \eqref{eq:conj22} has the following equivalent form about the distribution of fundamental units:
	for each $i \geq 0$ and $e\in \{0,1\}$,
	\begin{equation} \lim\limits_{x \rightarrow \infty}\frac{\#\{l\ \text{prime}:\ l\leq x,\ l \equiv (-1)^e \pmod8,\ \nu_2(a_l)=i+2+e\}}{\#\{l\ \text{prime}:\ l\leq x,\  l \equiv (-1)^e \pmod8\}}=\frac{1}{2^{i+1}}. \end{equation}

	(4) Finally, for $l\equiv 1\pmod{8}$,  \eqref{eq:conj22} actually has a finer form: for $i\ge 0$ and $a\in \{1,9\}$,
	\begin{equation} \label{eq:conj2a}	\lim\limits_{x \rightarrow \infty}\frac{\#\{l \leq x: l \equiv a \pmod {16},\ t_2(l)= 2^{i+1}\}}{\#\{l \leq x: l \equiv a \pmod {16}\}}=  \frac{1}{2^{i+1}}.	\end{equation}
	The cases $i=0$ and $1$ were proved in  Theorem~\ref{thm:i=0,1}. We actually speculate that this is the case for all sub-congruent classes $a\pmod{2^k}$ of $1\pmod{8}$.
	
	In the case $a=9$, let $\chi_l$ be  the associated Dirichlet character of $\BQ(\sqrt{l})$ and $L_2(s,\chi_l)$ be its $2$-adic $L$-function, by the $2$-adic class number formula (see \cite[Theorem 5.24]{Was97}) and Coates' order formula \eqref{eq:coates},   \eqref{eq:conj2a} has the following equivalent form which
	was implicitly proposed by Shanks-Sime-Washington in \cite[p. 1253]{SSW99}:
	\begin{equation}\label{eq:L1}\lim\limits_{x \rightarrow \infty} \frac{\#\{ l \text{ prime}: l \leq x,  l \equiv 9 \pmod {16} \text{ and } \nu_2(L_2(1,\chi_l))=i+2\}}{\#\{l \text{ prime}: l \leq x \text{ and }l \equiv 9 \pmod {16}\}}=\frac{1}{2^{i+1}}. \end{equation}

\end{remark}

\appendix
\section{Data for Conjecture~\ref{conj:CL-full-family}}
 In Tables~\ref{table:t5real}-\ref{table:t7ima}, we let the middle value be the ratio of field $F$ such that $\CT_p(F)\cong G$ among all  quadratic fields whose absolute discriminant $\leq B$, and $\BD$ be the value predicted by Conjecture~\ref{conj:CL-full-family}.

\begin{table}[H]
	\caption{$\CT_{5}$ of real quadratic fields}
	\centering
	\begin{tabular}{| c || c | c | c | c | c |}
		\hline
		\diagbox{$B$}{$G$}  & ${\BZ/5\BZ}$ & ${\BZ/25\BZ}$ & ${(\BZ/5\BZ)^2}$ & ${\BZ/5\BZ \times \BZ/25\BZ}$ &$ {(\BZ/5\BZ)^3}$\\
		\hline
		$10^7$  & 0.1876   & 0.03694  & 1.375 e-3 &	 3.277 e-4  & 0  \\
		$2*10^7$  &0.1880& 0.03712 &  1.396 e-3 &	3.463 e-4 & 1.645 e-7 \\
		$3*10^7$ & 0.1880 & 0.03727 & 1.416 e-3 &	3.439 e-4 & 2.193 e-7 \\
		$4*10^7$ & 0.1880 & 0.03739 & 1.438 e-3 &  3.447 e-4 &3.290 e-7 \\
		$5*10^7$  & 0.1882 & 0.03740 & 1.453 e-3 & 3.430 e-4 &2.632 e-7 \\
		\hline
		$\BD$  & 0.1901 & 0.03802  &1.584 e-3  & 3.802 e-4 & 5.110 e-7 \\
		\hline
	\end{tabular}
	\label{table:t5real}
\end{table}

\begin{table}[H]
	\caption{$\CT_{7}$ of real quadratic fields}
	\centering
	\begin{tabular}{| c || c | c | c | c |}
		\hline
		\diagbox{$B$}{$G$}  & ${\BZ/7\BZ}$ & ${\BZ/49\BZ}$ & ${(\BZ/7\BZ)^2}$  & ${(\BZ/7\BZ)^3}$\\
		\hline
		$10^7$  & 0.1377  & 0.01950  &  3.622 e-4&	 0    \\
		$2*10^7$   & 0.1382	&  0.01956	& 3.622 e-4 & 0 \\
		$3*10^7$   & 0.1383	& 0.01963	& 3.713 e-4 & 0 \\
		$4*10^7$   & 0.1385	&  0.01966	& 3.764 e-4  &  0 \\
		$5*10^7$  & 0.1385	& 0.01968	& 3.833 e-4  &  5.483 e-8 \\
		\hline
		$\BD$ &  0.1395  & 0.01992 & 4.151 e-4 &  2.477 e-8\\
		\hline
	\end{tabular} \label{table:t7real}
\end{table}

\begin{table}[H]
	\caption{$\CT_{5}$ of imaginary quadratic fields}
	\centering
	\begin{tabular}{| c || c | c | c | c | c |}
		\hline
		\diagbox{$B$}{$G$}  & ${\BZ/5\BZ}$ & ${\BZ/25\BZ}$ & ${(\BZ/5\BZ)^2}$ & ${\BZ/5\BZ \times \BZ/25\BZ}$ &$ {(\BZ/5\BZ)^3}$\\
		\hline
		$10^7$  &  0.04558 &  1.767 e-3 & 6.185 e-5  &	6.580 e-7 &0   \\
		$2*10^7$  & 0.04584 & 1.789 e-3 & 6.004 e-5 &	1.645 e-6 &0 \\
		$3*10^7$  & 0.04604 & 1.801 e-3 & 6.152 e-5 &	2.084 e-6 &0 \\
		$4*10^7$  & 0.04613 & 1.809 e-3 & 6.424 e-5 & 2.385 e-6 & 0\\
		$5*10^7$  & 0.04618 & 1.915 e-3 &  6.659 e-5 & 2.237 e-6 & 0\\
		\hline
		$\BD$ &   0.04752  & 1.901 e-3 & 7.920 e-5 & 3.802 e-6 & 5.110  e-9\\
		\hline
	\end{tabular} \label{table:t5ima}
\end{table}

\begin{table}[H]
	\caption{$\CT_{7}$ of imaginary quadratic fields}
	\centering
	\begin{tabular}{| c || c | c | c | c |}
		\hline
		\diagbox{$B$}{$G$}    & ${\BZ/7\BZ}$ & ${\BZ/49\BZ}$ & ${(\BZ/7\BZ)^2}$ & ${(\BZ/7\BZ)^3}$\\
		\hline
		$10^7$  & 0.02287  & 0.00043  &  3.619 e-6 &	0    \\
		$2*10^7$   &0.02297& 0.00045 & 5.263 e-6 & 0  \\
		$3*10^7$   & 0.02302 &0.00045 & 5.593 e-6 & 0  \\
		$4*10^7$   &0.02307 & 0.00045 & 6.827 e-6 & 0  \\
		$5*10^7$   & 0.02307 & 0.00045 &7.435 e-6 & 0  \\
		\hline
		$\BD$ &   0.02324  & 0.00047 & 9.883 e-6 & 8.425 e-11\\
		\hline
	\end{tabular} \label{table:t7ima}
\end{table}

\section{Data for Conjectures~\ref{conj:density1}}
In the last six Tables, we let the middle value be the ratio of field $F$ such that $\CT_p(F)\cong G$ among all  quadratic fields whose absolute discriminant $\leq B$, and $\BD$ be the value predicted by Conjecture~\ref{conj:density1}.

	\begin{table}[H]
		\centering
		\caption{$\CT_{2}$ of $\BQ (\sqrt{-l}), l\equiv 1\pmod {16} \text{ and } l \text{ is a prime }$}
		\begin{tabular}{|c||c|c|c|c|c|}
			\hline
			\diagbox{$B$}{$G$}      & $\BZ/8\BZ$ & $\BZ/16\BZ$ & $ \BZ/32\BZ$ &  $\BZ/64\BZ$  & $\BZ/128\BZ$\\
			\hline
			$10^7$   &  0.7508 & 0.1867 &  0.04704 & 0.01172 &  2.905 e-3    \\
			\hline
			$2*10^7$   & 0.7501  &  0.1872 &  0.04708 & 0.01170 &  3.062 e-3      \\
			\hline
			$3*10^7$  &  0.7501 & 0.1878& 0.04658  &0.01169  &2.977 e-3    \\
			\hline
			$4*10^7$   & 0.7498 &  0.1881 & 0.04666  & 0.01166 &  2.910 e-3     \\
			\hline
			$5*10^7$  & 0.7496  &0.1880 & 0.04694  & 0.01160 &  2.934 e-3  \\
			\hline
			$\BD$ &  0.75 &  0.1875 & 0.04688  & 0.01172  &   2.930 e-3  \\
			\hline
		\end{tabular} \label{table:-l116}
\end{table}

	\begin{table}[H]
		\centering
		\caption{$\CT_{2}$ of $\BQ (\sqrt{-2l}), l\equiv 1\pmod {16} \text{ and } l \text{ is a prime }$}
		\begin{tabular}{|c||c|c|c|c|c|} \hline
			\diagbox{$B$}{$G$}      & $\BZ/4\BZ$ & $\BZ/8\BZ$ & $\BZ/16\BZ$ & $ \BZ/32\BZ$ &  $\BZ/64\BZ$  \\
			\hline
			$10^7$   &  0.7508 & 0.1876  & 0.04611& 0.01144 &  3.134 e-3      \\
			\hline
			$2*10^7$  &  0.7501 & 0.1886& 0.04604  &0.01142  &  3.075 e-3    \\
			\hline
			$3*10^7$   & 0.7501  & 0.1885 &  0.04611 &0.01140 &   3.029 e-3    \\
			\hline
			$4*10^7$  &0.7498  & 0.1885& 0.04633  & 0.01153 &  3.032 e-3  \\
			\hline
			$5*10^7$  & 0.7496 & 0.1883 &0.04655  & 0.01157 &  3.051 e-3  \\
			\hline
			$\BD$ &  0.7500 &  0.1875 & 0.04688  & 0.01172  &  2.930 e-3  \\
			\hline
		\end{tabular}  \label{table:-2l116}
\end{table}
	\begin{table}[H]
		\centering
		\caption{$\CT_{2}$ of $\BQ (\sqrt l), l\equiv 1\pmod 8 \text{ and } l \text{ is a prime }$}
		\begin{tabular}{|c||c|c|c|c|c|c|}
			\hline
			\diagbox{$B$}{$G$}      & $\BZ/2\BZ$ & $\BZ/4\BZ$ & $\BZ/8\BZ$ & $ \BZ/16\BZ$ &  $\BZ/32\BZ$ & $\BZ/64\BZ$ \\
			\hline
			$10^7$   & 0.5002  & 0.2499  & 0.1245  & 0.06236 & 0.03169  &  0.01553  \\
			\hline
			$2*10^7$   & 0.5000  & 0.2499  & 0.1245  &  0.06255 &   0.03163  &0.01567   \\
			\hline
			$3*10^7$  & 0.5005  & 0.2496  & 0.1246  &  0.06278 &  0.03115  &  0.01560\\
			\hline
			$ 4*10^7$  & 0.5003  & 0.2496  & 0.1247  & 0.06278  &   0.03115 &    0.01564  \\
			\hline
			$5*10^7$  & 0.5001  & 0.2497  & 0.1247  &  0.06281 & 0.03116   &   0.01567   \\
			\hline
			$\BD$ &  0.5000 &  0.2500 & 0.1250  & 0.06250  &  0.03125 & 0.01563   \\
			\hline
		\end{tabular}  \label{table:l18}
\end{table}
	\begin{table}[H]
		\centering
		\caption{$\CT_{2}$ of $\BQ (\sqrt l), l\equiv 7\pmod 8 \text{ and } l \text{ is a prime }$}
		\begin{tabular}{|c||c|c|c|c|c|c|}
			\hline
			\diagbox{$B$}{$G$}      & $\BZ/4\BZ$ & $\BZ/8\BZ$&$ \BZ/16\BZ$ &  $\BZ/32\BZ$ & $\BZ/64\BZ$ & $\BZ/128\BZ$  \\
			\hline
			$10^7$    & 0.5000   & 0.2484   & 0.1260  &  0.06361   & 0.03103   &0.01518    \\
			\hline
			$ 2*10^7$     & 0.5000   & 0.2494   & 0.1255   &  0.06265  & 0.03123  &  0.01534  \\
			\hline
			$3*10^7$    & 0.4998   & 0.2497   & 0.1252   &   0.06278   & 0.03109  & 0.01557\\
			\hline
			$4*10^7$  & 0.4999   & 0.2497   &  0.1254  & 0.06246    &   0.03112  & 0.01570  \\
			\hline
			$ 5*10^7$   & 0.5001   & 0.2497   & 0.1254   &  0.06237   & 0.03116  & 0.01570   \\
			\hline
			$\BD$ &  0.5000  &  0.2.500  & 0.1250   & 0.06250   &  0.03125  & 0.01563   \\
			\hline
		\end{tabular}   \label{table:l78}
\end{table}
	\begin{table}[H]
		\centering
		\caption{$\CT_{2}$ of $\BQ (\sqrt {2l}), l\equiv 1\pmod 8 \text{ and } l \text{ is a prime }$}
		\begin{tabular}{|c||c|c|c|c|c|c|}
			\hline
			\diagbox{$B$}{$G$}     & $\BZ/2\BZ$ & $\BZ/4\BZ$ & $\BZ/8\BZ$ & $\BZ/16\BZ$ &  $\BZ/32\BZ$ & $\BZ/64\BZ$\\
			\hline
			$10^7$   &  0.5006   &  0.2515   &  0.1237  & 0.06214  & 0.03100   & 0.01564  \\
			\hline
			$2*10^7$   & 0.5004   & 0.2511   &  0.1239   & 0.06219  &  0.03105   &  0.01576     \\
			\hline
			$3*10^7$  & 0.5001   &  0.2506  &  0.1245  &   0.06256  &   0.03093   &    0.01572    \\
			\hline
			$ 4*10^7$  &  0.5001  & 0.2505   &  0.1249  &  0.06233  &  0.03090  &   0.01564  \\
			\hline
			$5*10^7$  &  0.5000 & 0.2503  & 0.1252   & 0.06236  &    0.03083  &   0.01572       \\
			\hline
			$\BD$ &  0.5000  &  0.2500  & 0.1250   & 0.06250   &  0.03125  & 0.01563   \\
			\hline
		\end{tabular}  \label{table:2l18}
\end{table}

\begin{small}
	\begin{table}[H]
		\centering
		\caption{$\CT_{2}$ of $\BQ (\sqrt {2l}), l\equiv 7\pmod 8 \text{ and } l \text{ is a prime }$}
		\begin{tabular}{|c||c|c|c|c|c|c|}
			\hline
			\diagbox{$B$}{$G$}     & $\BZ/2\BZ$ & $\BZ/4\BZ$ & $\BZ/8\BZ$ & $\BZ/16\BZ$ &  $\BZ/32\BZ$ & $\BZ/64\BZ$ \\
			\hline
			$10^7$   &  0.5000   &  0.2484  & 0.1253   & 0.06378  & 0.03129   & 0.01565       \\
			\hline
			$2*10^7$   & 0.5000   & 0.2.494   &  0.1253   &  0.06258  & 0.03137    &  0.01565     \\
			\hline
			$3*10^7$  & 0.4998   & 0.2497  &   0.1254  &  0.06258  &  0.03116  &   0.01575      \\
			\hline
			$ 4*10^7$  & 0.4999   & 0.2497   & 0.1252   &  0.06267  & 0.03129    &  0.01569       \\
			\hline
			$5*10^7$  &  0.5001  &  0.2497  &  0.1250  &  0.06268  &  0.03126   &   0.01573   \\
			\hline
			$\BD$ &  0.5000  &  0.2500  & 0.1250   & 0.06250   &  0.03125  & 0.01563   \\
			\hline
		\end{tabular}  \label{table:2l78}
\end{table}\end{small}

\end{document}